\documentclass{article}
\usepackage[utf8]{inputenc}
 \usepackage[a4paper]{geometry}

\RequirePackage{amsmath,amsthm,natbib}
\RequirePackage[colorlinks,citecolor=blue,urlcolor=blue]{hyperref}
\usepackage{enumerate}
\usepackage{graphicx}
\usepackage{amsfonts,amssymb}
\usepackage{dsfont}

\newtheorem{theorem}{Theorem}

\newtheorem{corollary}{Corollary}
\newtheorem{proposition}{Proposition}
\newtheorem{lemma}{Lemma}

\theoremstyle{plain}
\theoremstyle{remark}

\newcommand\ind{\mathds{1}}     
\DeclareMathOperator{\Var}{Var}

\DeclareMathOperator{\cov}{cov}

\newcommand{\argmin}{\operatornamewithlimits{\arg\min}}

\newcommand{\diff}{\mathrm{d}}

\begin{document}

\title{Nearest neighbor empirical processes}

\author{{Fran\c{c}ois} {Portier}\footnote{email: francois.portier@gmail.com}\\
ENSAI, CREST, University of Rennes}






%

%
\maketitle

\begin{abstract}
{In the regression framework, 
the empirical measure based on the responses resulting from the nearest neighbors, among the covariates,}
to a given point $x$ is introduced and studied as a central statistical quantity. First, the associated empirical process is shown to satisfy a uniform central limit theorem under a local bracketing entropy condition on the underlying class of functions reflecting the localizing nature of the nearest neighbor algorithm. 
Second a uniform non-asymptotic bound is established under a well-known condition, often referred to as Vapnik-Chervonenkis, on the uniform entropy numbers. The covariance of the Gaussian limit obtained in the uniform central limit theorem is simply equal to the conditional covariance operator given the covariate value. 
This suggests the possibility of using standard formulas to estimate the variance by using only the nearest neighbors instead of the full data.
%
{This is illustrated on two problems: the estimation of the conditional cumulative distribution function and local linear regression.}\\

\noindent\textbf{Primary} {62G05}, 62G08; \textbf{secondary} {62G20}\\
\noindent\textbf{Keywords:} Concentration inequality, empirical process theory, nearest neighbor algorithm, weak convergence

\end{abstract}

\section{Introduction}

The nearest neighbor algorithm is one of the most intuitive yet powerful nonparametric regression methods used to predict an output based on some features or covariates. While the resulting estimator is known to match the optimal convergence rate for nonparametric regression \citep{gyorfi1981rate,jiang2019non}, it also takes advantage of the local intrinsic structure of the feature space \citep{kpotufe2011k}. In contrast to the Nadaraya-Watson alternative \citep{nadaraya1964estimating,watson1964smooth}, the nearest neighbor method is adaptive to the features probability density function, as a large bandwidth is automatically set in low density region, making the asymptotic variance independent of the probability density value \citep{mack1981local}. The method is flexible because it can easily be improved by some metric learning procedures \citep{weinberger2006distance},  be implemented in parallel on different computers, allowing for its use in large-scale learning problems \citep{devroye1980consistency,qiao2019rates}, be used to estimate residual variance \citep{devroye2018nearest} or sparse gradients \citep{ausset2021nearest} in regression models.

{
Given a collection of $n\geq 1$ elements $(X_1, Y_{1} ) ,\ldots ,(X_n, Y_{n})  $, where $X_i\in  \mathbb R^d$ is the covariate and $Y_i \in \mathbb R$ is the response, and a parameter $1\leq k \leq n$, the $k$-nearest neighbor ($k$-NN for short)  regression estimate at point $x\in \mathbb R^d$  is defined as
$$ k^{-1} \sum_{i  \in   N_{n,k}(x) } Y_i ,$$
where  $ N_{n,k}(x)$ is the index set of the $k$-nearest neighbor to point $x$ among $X_1,\ldots , X_n$ (if there are ties use the lexicographic order).} The $k$-NN algorithm was initially introduced in \cite{fix1989discriminatory,royall1966class,cover1967nearest}  and has since been the subject of many studies in the statistical and machine learning literature. Pointwise, uniform and $L_1$ consistency of $k$-NN regression are respectively studied in \cite{stone1977consistent}, \cite{devroye1978uniform} and \cite{devroye1994strong}. The rate of convergence (in probability) is investigated in \cite{gyorfi1981rate} and an asymptotic normality result is established in \cite{mack1981local}. 
More recently, bounds on the $L_2$-risk have been obtained for a bagged version of the $k$-NN regression rule \citep{biau2010rate}. The excess error probability associated with the $k$-NN classifier have been examined in \cite{doring2017rate,gadat2016classification,cannings2020local} where a particular attention is given to unbounded covariates domains. Non-asymptotic uniform concentration bounds, matching the optimal nonparametric convergence rate \citep{stone1982optimal}, have been derived for $k$-NN regression \citep{jiang2019non}. Other works consider intrinsic properties of $k$-NN such as the size and volume of the $k$-NN balls \citep{devroye1977strong}  or the $k$-NN cells, i.e., the elements of the Voronoi partition resulting from the covariates sample \citep{devroye2017measure}.  
For a broader literature review dealing with $k$-NN density estimators, regression and classification, we refer to the textbooks \cite{gyorfi2006distribution} and \cite{biau2015lectures}.


{
 The object of interest in this work is called the $k$-\textit{NN empirical measure} and is defined as
\begin{align*}
\hat  \mu_{n,k,x} ( A ) :  =k ^{-1}    \sum_{i \in  N_{n,k}(x)} \ind _ A(Y_i)    ,\qquad   A \in \mathcal S, 
\end{align*}
where, from now on, $Y_i$ will be valued in $(S,\mathcal S)$, a general measurable space. Following the rationale of \textit{local averaging} lying behind the nearest neighbor selection procedure \citep{gyorfi2006distribution}, the $k$-NN empirical measure is an estimate of $\mu_x$, the (regular) conditional measure of $ Y$ given $X = x$ (which further is supposed to exist). 
The error between two measures will be evaluated using some class $\mathcal G$  of real valued measurable functions on $(S,\mathcal S)$.
For any such class, the $k$-NN measure induces a (random) map defined by
$ \hat  \mu_{n,k,x} ( g ) : = \int g \, \diff \hat \mu_{n,k,x}$,  for any $g\in\mathcal G$. We have the following bias-variance decomposition
\begin{align*}
  ( \hat \mu _{n,k,x}  - \mu _x) (g ) &=   k ^{-1} \sum_{i \in  N_{n,k}(x)}   (  g(Y_i )   -\mu_{X_i} (g)   )  + k ^{-1} \sum_{i \in  N_{n,k}(x)}   (\mu_{X_i} (g)   - \mu_x (g) ) ,
\end{align*}
where the first term will be called the $k$-NN process and the second term shall be called the bias term.
Although some upper-bounds will be obtained for the bias term, the most difficult part, and therefore the main contribution of the paper, will be to examine the $k$-NN process which carries-out the stochastic variability of the estimation error. Remark that the previous definition is comparable to that of standard empirical processes \citep{wellner1996} though a square-root of $k$ instead of a square-root of $n$ will be needed as a normalizing factor for weak convergence. Note that completely different statistical processes with similar names were introduced in \cite{bickel1983sums} for goodness-of-fit purpose as well as in \cite{datta2016hierarchical} for application in spatial statistics.


}

From a practical point of view, the $k$-NN measure is of great interest because many meaningful statistical objects can be examined through it. The standard $k$-NN regression rule takes the form $\int y \, \diff \hat  \mu_{n,k,x}(y)$ while alternative robust versions might be expressed as $ \min_{\theta } \hat \mu_{n,k,x} (\psi_\theta)$, with for instance $\psi_{\theta} (y) = |y-\theta|$, as studied in \cite{hardle1988robust} when using Nadaraya-Watson weights. Moreover, the $k$-NN estimate of the conditional distribution function is given by 
$\hat F_{n,k,x} (y) = \hat  \mu_{n,k,x} ( \ind_{(-\infty , y] } ) $, $ y\in {\mathbb R}$,
from which we can easily deduce $k$-NN quantiles $ u \mapsto \hat F_{n,k,x} ^-  (u)$ where $F^-$ stands for the generalized inverse of $F$. Similarly, $k$-NN versions of conditional copula \citep{veraverbeke2011estimation}, Nelson-Aalen or Kaplan-Meier \citep{beran1981nonparametric,dabrowska1989uniform} estimates might be obtained from applying the corresponding transformation to the $k$-NN measure. A summary of  several such useful transformations is given in \citep[Section 3.9.4]{wellner1996}.

Two main results will be obtained in the paper:
\begin{itemize}
\item [(i)] \textit{\textit{Uniform central limit theorem}.}  Under some conditions on $\mathcal G$, we show that the $k$-\textit{NN process}, when multiplied by square-root $k$, converges weakly to a Gaussian process with covariance function being simply the conditional covariance of $Y$ given $X = x$. This property is obtained under a condition on the bracketing numbers of the space $(\mathcal G, L_2(Q))$ for some probability measure $Q$. Our condition bears resemblance to usual bracketing entropy bounds useful to obtain  Donsker's type theorem \citep[Theorem 2.11.23]{wellner1996} except that our bound needs to hold uniformly in $Q$ belonging to some neighbourhood of the conditional measure $\mu_x$. This might be seen as the price to pay for not observing data distributed under $\mu_x$. Finally, as the Gaussian limiting process has a particularly simple covariance function (contrary to that of the Nadaraya-Watson-like measure), it suggests easy ways to make inference extending available variance formula or bootstrap procedures valid for the standard empirical process to the local $k$-NN process. {This will be discussed considering the conditional cumulative distribution function estimator as well as the local linear regression estimator, both resulting from the $k$-NN measure.}
\item [(ii)] \textit{\textit{Uniform non-asymptotic upper-bound}.} 
 To complement the previous asymptotic result, we obtain a non-asymptotic bound on the $k$-NN process holding uniformly over $g\in \mathcal G$ and $x\in \mathbb R^d$. The main working condition is a standard Vapnik-Chervonenkis (VC)-type assumption \citep{nolan+p:1987,gine+g:02}  on the covering numbers of the class $\mathcal G$ allowing us {to use Dudley's entropy integral bound \citep{dudley1967sizes} coupled with Bousquet's version \citep{bousquet2002bennett} of Talagrand's inequality \citep{talagrandNewConcentrationInequalities1996}}. The upper-bound that is obtained is sharp as it matches (up to a $\log(n)$ factor) the optimal rate of convergence for nonparametric regression.  {Moreover, the constants involved in	 the bound are either explicit or well identified elements such as the dimension $d$,  the number of samples $n$, the number of neighbors $k$, or the VC parameters of the class $\mathcal G$.}
 
\end{itemize}

Our uniform central limit theorem is a functional version of the central limit theorem given in \cite{mack1981local} or in \cite{biau2015lectures} as Theorem 14.4. It bears resemblance with the results provided in \cite{stute1986conditional,horvath1988asymptotics} in which a certain conditional empirical process is introduced and studied. Their conditional empirical process corresponds to the above $k$-NN process defined on the class $\mathcal G = \{g(Y) = \ind _{\{Y\leq y \}}\, : \, y\in \mathbb R\}$ making our framework more general. In addition, their proof technique allows only to cover the case $d= 1$ due to the use of an ordering of the covariates values. Our uniform non-asymptotic upper-bound extends recent uniform and non-asymptotic results in \cite{jiang2019non} into an empirical process version. Note also \cite{hardle1988strong,einmahl2000empirical,hansen2008uniform} where uniform asymptotic bounds are obtained for Nadaraya-Watson type estimators.

As a secondary result, we obtain upper-bounds on the bias term reflecting two situations regarding the regularity of the $k$-NN measure. When the map $x\mapsto \mu_x(g)$ is Lipschitz only, then the bias term is of order $(k/n) ^{1/d}$ but when the  map $x\mapsto \mu_x(g)$ is twice-differentiable with bounded second-order derivatives then the estimation bias improves compared to before as the rate becomes $(k/n) ^{2/d}$. Those facts show that the $k$-NN algorithm takes advantage of the smoothness of the target measure. For point-wise $k$-NN regression, i.e., $g(y) = y$ and $x$ is fixed, such a result can be found in \cite{biau2015lectures}, Section 14.3. Note that the uniform bound in \cite{jiang2019non} only relies on Lipschitz regularity of the regression function and therefore cannot benefit from second-order regularity properties.

\textbf{Outline.} 
The main definitions and concepts (along with some notation) are introduced in Section \ref{sec:def}. The uniform central limit theorem, the upper-bounds on the bias term and some illustrative examples are given in Section \ref{sec:weak_cv}. The uniform non-asymptotic bound (along with some remarks) is provided in Section \ref{sec:uniform_bound}.
Based on some technical lemmas, proofs of which  are given in the Appendix, the proofs of the main results (Theorem \ref{weakcv:kNN}, \ref{bias_bound} and \ref{consistenc:kNN}) are presented in Section \ref{section:proofs}. 

\section{Mathematical background and notation} \label{sec:def}
{
In this section, the probabilistic framework of the paper is introduced. Several well-known concepts such as bracketing numbers, covering numbers and VC classes are described and illustrated  with the help of two standard theorems that will be used in the proofs of the paper.
}

\subsection{Probabilistic framework}

All the results of the paper are based on the following assumption which requires the random variables of interest to be independent and identically distributed. We shall also assume throughout that $X$ has a density with respect to the Lebesgue measure $\lambda$ defined on $\mathbb R^d$.

\begin{itemize}
\item[(A)] Let $(S,\mathcal S)$ be a measurable space. Let $(X,Y) , (X_i, Y_{i} )_{i \geq 1 }$ be a sequence of independent and identically distributed random elements with common distribution $P$ on $( \mathbb R^d \times S ,\mathcal B ( \mathbb R^d)  \otimes \mathcal S) $. Let $P_X$ denote the probability measure of $X$. We have for all $B\in \mathcal B ( \mathbb R^d)$,  $P_X(B) = \int _B f_X  \diff \lambda$ where $f_X$ is called the density function of $X$.
\end{itemize}


 We now introduce  $\mu_x$ as the regular conditional distribution of $Y | X =x $ (which further is supposed to exist), that is:
\begin{itemize}
\item  for $P_X$-almost every $x \in \mathbb R^d$,  $\mu_x$ is a probability measure on $(S,\mathcal S)$;
\item for all $A\in \mathcal S$, $x\mapsto \mu_x(A)$ is measurable;
\item  for all $A\in   \mathcal S $, $B\in   \mathcal B ( \mathbb R^d)  $,  $  P [ (Y,X)\in A\times B ]   
=\int_B \mu_x(A)     f_X (x)\,  \diff x$.
\end{itemize}
The space $\mathbb R^d$ is endowed with the Euclidean  norm $\|\cdot \|$. For $x\in \mathbb R^d$ and $\tau >0$,  $ B(x,\tau)$ denote the set of all points $z\in \mathbb R^d $ satisfying $\|x-z \|\leq \tau$. The unit ball volume is $V_d  = \lambda (B(0,1) )$. The quantity $\ind _A (x)$ is $ 1$ if $x\in A$ and $0$ elsewhere. For any measure $\mu$ on $\mathcal S$ and real function $g$ defined on $S$, 
$\mu(g) = \int g\diff \mu$ and $\Var_\mu (g) = \mu(g^2) - \mu(g) ^2$. For any set $T$, $\ell^{\infty}(T) $ denote the space of all uniformly bounded real functions defined on $T$. 
For any sequence of random variables $(Z_n)_{n\geq 1} $ and any sequence of positive numbers $(a_n)_{n\geq 1}$, we write $Z_n = o_P(a_n)$  if $ P (Z_n  > a_n \epsilon) \to 0$ for any $\epsilon>0$. We write $Z_n=O_P(a_n) $ if for any $\epsilon >0 $, there is $A>0$, such that
$ P (Z_n   > A a_n ) \leq \epsilon $.

\subsection{Bracketing numbers}
{

Given a probability measure $Q$ on $(S,\mathcal S)$, the metric space of squared-integrable functions with respect to $Q$ is defined as
\begin{align*}
L_2(Q) = \left\{ g: S\mapsto \mathbb R \text{ such that  } \| g\|_{L_2(Q)} ^ 2 :=  Q(g^2)  <  \infty \right\} . 
\end{align*} 
Let $\underline{ f}$ and $ \overline{f}$ be two functions in $L_2(Q)$. The set $[\underline{ f}, \overline{f}]$, called a bracket, denotes the set of all functions $g $ in $L_2(Q)$ such that $\underline{ f} \leq g\leq  \overline{f} $. A bracket $ [\underline{ f}, \overline{f}]$ such that $\| \underline{ f} -  \overline{f} \| _{L_2(Q)} \leq \epsilon$ is called an $\epsilon$-bracket. Given $\mathcal G \subset L_2(Q)$, the $\epsilon$-bracketing number, denoted $\mathcal N _{[\,]} (\mathcal G , L_2(Q) , \epsilon)$, is defined as the smallest number of brackets of radius $\epsilon > 0$ needed to cover $\mathcal G$. We next call an envelope for $\mathcal G$ any function $G:S\mapsto \mathbb R$ that satisfies $|g(x)| \leq G(x)$ for all $x\in S$ and $g\in \mathcal G$. Those envelope functions are useful to normalize bracketing numbers such that the following bracketing integral
$$ J(\mathcal G,  L_2(Q),\delta ) =  \int _ 0 ^{\delta}   \sqrt { \log\left(  \mathcal N_{[\,]} \left(\mathcal G,  L_2(Q)  ,  \epsilon \| G \| _ {L_2(Q)}   \right) \right)} \,\diff \epsilon$$
is constant as soon as $\delta\geq 2$ as a single bracket, $[-G,G]$, can be used to cover $\mathcal G$ with size $ 2\|G\| _ {L_2(Q )} $. 

Bracketing numbers will be key to obtain, in Section \ref{sec:weak_cv}, the weak convergence property of the $k$-NN process. Note that bracketing numbers are known to be a powerful tool to prove the weak convergence of the standard empirical process. We now recall \citep[Theorem 2.11.23]{wellner1996} which is useful for class of functions that might depend on $n$.

\begin{theorem}\label{vdvw_weakcv}
Let $(Z_i)_{i\geq 1}$ be a sequence of independent and identically distributed random variables with common distribution $P$. For each $n$, let $\mathcal F_n = \{f_{n,t} \, : \,  t\in T\}$ be a class of measurable functions, with measurable envelope $F_n$, indexed by a totally bounded semimetric space $(T, \rho)$. Suppose that $( PF_n ^2)_{n\geq 1}$ is bounded and for every positive sequence $\delta_n \to 0$ and $\eta >0$,
	\begin{align*}
	& \limsup _{n\geq 1} P(F_n^2 \ind _ { F_n > \eta \sqrt n   } ) = 0,\\
& \limsup _{n\geq 1}  \sup_{\rho(t,s)\leq \delta_n } P (f_{n,t} - f_{n,s} ) ^2  = 0  ,\\ 	
& \limsup _{n\geq 1}  J(\mathcal F_n,  L_2(P),\delta _n ) = 0 .
	\end{align*}
Then, the sequence $\{ (1/  \sqrt n)   \sum_{i=1} ^n   \{ f_{n,t} (Z_i )  - P (f_{n,t}) ] \}  : t \in  T\} $  is asymptotically tight in $\ell ^ \infty (T) $  and converges in distribution to a tight Gaussian process provided the sequence of covariance functions $P ( f_{n,t} f_{n,s}  ) - P f_{n,t}  P f_{n,s} $ converges pointwise on $ T\times T$.
\end{theorem}


An important example of classes of functions that satisfies the above condition on the bracketing number is the class of cells $ \mathcal I = \{\mathds 1_{\{\cdot\leq y \}}$, $y\in \mathbb R\} $, for which $\mathcal N _{[\,]} (\mathcal I , L_2(Q) , \epsilon)\leq 4 /\epsilon^2 $ as shown in  Example 2.5.4 in \citep{wellner1996}. As a consequence, the use of bracketing numbers allows to recover the well-known Donsker theorem which claims the weak convergence of the re-scaled sequence of empirical cumulative distribution functions.

}

\subsection{Covering numbers and VC classes}

{
Given $\mathcal G \subset L_2(Q)$, the $\epsilon$-covering number, denoted $\mathcal N (\mathcal G , L_2(Q) , \epsilon)$, is defined as the smallest number of closed balls of radius $\epsilon > 0$ needed to cover $\mathcal G$. 
A class $\mathcal G$ of pointwise measurable \citep[Example 2.3.4]{wellner1996} real-valued functions on a measurable space $(S, \mathcal S)$ is said to be VC with parameters $(v, A) \in (0, \infty)\times [1, \infty)$ and envelope $G$ if for any $0 < \epsilon < 1$ and any probability measure $Q$ on $(S, \mathcal S)$, we have
	\begin{equation*}
	\mathcal N \left(\mathcal G,  L_2(Q)  ,  \epsilon \| G \| _ {L_2(Q) } \right) \le (A/\epsilon)^{v}.
	\end{equation*}
The class of cells $\mathcal I$ defined above is VC with parameter $v = 2$ \citep[Example 2.5.4]{wellner1996}. Another example, which  will be central to our study, is the class of Euclidean balls. As shown in Lemma \ref{vc:pres3} given in the Appendix, this class is VC with parameter $v = 2( d+1) $.

 VC classes will be used, in Section \ref{sec:uniform_bound}, to obtain a non-asymptotic upper-bound on the $k$-NN process. Note that this concept has already been employed to derive excess risk bounds for empirical risk minimizer \citep{boucheron2005theory,bartlett2005local} or uniform upper-bounds on the estimation error for kernel regression estimators \citep{nolan+p:1987,gine+g:02,gineConsistencyKernelDensity2001} or for multiple ordinary least-squares procedures \citep{plassier2023risk}. While different definitions of VC classes exist, here we rely on the definition used in the previous references which is based on the covering numbers. This notion is more general than the one initially introduced in \cite{vapnik2015uniform} as explained for instance  in \cite[Theorem 2.6.4]{wellner1996}.

The next non-asymptotic concentration inequality is tailored to VC classes of functions. The proof, given in the Appendix, relies on  Bousquet's version \citep{bousquet2002bennett} of Talagrand's inequality \citep{talagrandNewConcentrationInequalities1996} and Dudley's entropy integral bound \citep{dudley1967sizes}.

\begin{theorem}\label{vc_bound}
Let $(Z_1,\ldots ,Z_n) $ be an independent and identically distributed collection of random variables  with common distribution $P$ on $ (S,\mathcal S)$. Let $\mathcal G$ be a VC class of functions with parameters $v\geq 1 $, $A\geq 1$ and uniform envelope $U\geq \sup_{g\in \mathcal G,\, x\in S}  |g(x)|$. Let $\sigma$ be such that $\sup_{g\in \mathcal G} \Var_P(g) \leq \sigma^2  \leq U^2 $.  For any $n\geq 1$ and $\delta\in (0,1) $, it holds, with probability at least $1-\delta$,
\begin{align*}
&  \sup_{g\in \mathcal G} \left| \sum_{i=1} ^n   \{g (Z_i )  -  P (g )  \} \right|  \leq  K_1 
 \sqrt{  vn \sigma^2   \log\left(  9AU  /  (\sigma \delta)    \right)      }  +   K_2   Uv   \log\left(  9AU  /  (\sigma \delta)     \right)  ,
  \end{align*} 
	with $K_1 =    5C   $,  $K_2 =   64 C^2    $ and  $C=  9 $. 
\end{theorem}

A similar inequality is obtained in \cite{gineConsistencyKernelDensity2001}, Proposition 2.2, but using universal unknown constants in place of the  above explicit constants $K_1$ and $K_2$. 
Another related inequality, with explicit constants, valid for the expected value of such supremum and for functions satisfying $P(g) = 0$, can be found in \cite{gine2009}, Proposition 3.


}

\section{Weak convergence under bracketing entropy}\label{sec:weak_cv}

\subsection{Statement of the result}\label{subsec:weak_cv}


The weak convergence property of the $k$-NN process is obtained under the following metric entropy condition. For any $u > 0$, define the probability measure
\begin{align*}
 \mu_{x, u} ( A) := \frac{ E (  \mu_X (A)   \ind _{  B(x , u^{1/d} )  }  (X) )  }{  E (   \ind _{  B(x , u^{1/d} )  }  (X) )   }  ,\qquad A\in \mathcal  S.
\end{align*}
For some $\delta>0$ and for any positive sequence $(\delta_n)_{n\geq 1}$ going to $0$, it holds that
	\begin{equation}\label{def:brack_cond}
\limsup _{n\geq 1}   \sup_{ |u|\leq  \delta   } J( \mathcal G,  L_2(\mu_{x,u} )  ,  \delta_n)  =  0, 
	\end{equation}
	where $J$ is the bracketing integral introduced in the previous section.
Another assumption, known as the Lindeberg condition, is standard to obtain the weak convergence of the empirical process. In our setting, the Lindeberg condition takes the form of a uniform integrability assumption with respect to a given neighborhood of the measure $\mu_x$. More precisely, we ask that for some $\delta>0$ and an envelope $G$ for the class $\mathcal G$: for all $\epsilon $, there is $M > 0$ such that
\begin{align}\label{def:ui_local}
\sup _{| u| \leq \delta  } \mu_{x,u} ( G^2 \ind _ {G > M })  \leq \epsilon.
\end{align} 
We can now state our weak convergence result using 
the semimetric $\rho_\delta ( g_1,g_2) = \sup_{ |u| \leq \delta}   \|  g_1 - g_2 \|_{L_2(  \mu_{x, u})}   $. The proof of the next theorem is given in Section \ref{section:proofs} up to two technical Lemmas, namely Lemma \ref{prop:tau_x} and Lemma \ref{lemma:brecketing_cond}, that are proved in the Appendix.

\begin{theorem}\label{weakcv:kNN}
{Suppose that (A) is fulfilled.
Let $x\in \mathbb R^d$ and assume that $f_X$ is 
continuous at $x$ and $f_X(x) >0$.
Suppose there is $\delta >0 $ and an envelope $G$ for the class $\mathcal G$ such that \eqref{def:brack_cond} and \eqref{def:ui_local} are satisfied and $(\mathcal G,\rho_\delta )$ is totally bounded. Suppose also that for any pair $(g_1,g_2)$ in $\mathcal G$, the map $z\mapsto (\mu_z(g_1g_2) , \mu_z(g_1)  ,    \mu_{z} (G^2) )  $ is continuous at $x$. If, in addition, $k:=k_n$ satisfies $   k \to \infty$ and $k/n \to 0 $, it holds that
\begin{align*}
\left\{  k ^{-1/2} \sum_{i \in  N_{n,k}(x)}   (  g(Y_i )   -\mu_{X_i} (g)   )\right\}_{g\in \mathcal G}  \text{  converges weakly in $\ell^{\infty}(\mathcal G)$ to a Gaussian process}
\end{align*}
  with covariance function $\cov_x: (g_1,g_2) \mapsto \mu_x(g_1 g_2) -\mu_x(g_1 ) \mu_x (g_2) $.
}

\end{theorem}

\subsection{Discussion and remarks}\label{subsec:weak_cv_rk}

\textbf{Sketch of the proof.} For $n\geq 1$ and $k\in\{1,\; \ldots,\; n\}$, the $k$-NN bandwidth at $x$ is denoted by  $\hat \tau_{n,k,x} $ and defined as the smallest radius $\tau \geq  0$ such that the ball $ B(x,\tau)$ contains at least $k$ points from the collection $\{ X_1,\ldots,X_n\}$. That is,
\begin{align*}
\hat \tau_{n,k,x} : =  \inf   \{ \tau\geq 0 \, :\,  \sum_{i=1}^n \ind _{ B(x,\tau) }(X_i) \geq k \} .
\end{align*}
In the proof, we use the  $k$-NN bandwidth to write the $k$-NN measure as
\begin{align*}
\hat  \mu_{n,k ,x} ( A ) =  k^{-1} \sum_{i=1}^n \ind _ A(Y_i) \ind _{  B(x,  \hat \tau_{n,k,x}   )  }(  X_i).
\end{align*}
To obtain an accurate non-random estimate $  \tau_{n,k,x}$ of $\hat \tau_{n,k,x} $, one may leverage the empirical process equicontinuity \citep{wellner2007} to guarantee that $\hat  \mu_{n,k ,x} $ has the same behaviour as a simpler measure, defined in the same way, but using $ \tau_{n,k,x}$ instead of $\hat \tau_{n,k,x} $. To conclude the proof, it remains to establish the desired weak convergence property for this simpler measure. This remaining point will follow from using weak convergence results for functions classes that might change with $n$ as in Theorem \ref{vdvw_weakcv} given in Section \ref{sec:def}.

\textbf{The limiting covariance.}
The covariance of the limiting Gaussian proceess is particularly simple and meaningful: while the standard empirical process limiting covariance has the form $\cov(g_1,g_2)$ where the covariance is associated with the underlying probability distribution, the limiting covariance obtained in Theorem \ref{weakcv:kNN} is just the same but with the conditional covariance operator (given $X = x$) hence switching from the full distribution to a localized version. One important practical implication is that formulas for the variance of statistics resulting from the $k$-NN measure might be deduced easily from existing formulas valid for statistics obtained from the standard empirical measure.
{More precisely, we support the following plug-in principle: variance estimates should be computed similarly to variance estimates of standard approaches, but instead of computing them with respect to the full data, one would only use  the $k$-nearest neighbors (two examples are given in the next section). }


\textbf{Comparison to the Nadaraya-Watson estimator.}
Also known as kernel smoothing regression, the Nadaraya-Watson approach is similar to the one of $k$-NN in that the resulting estimator writes as a weighted sum with weights that decrease with the distance to the point of interest. Given a bandwidth $\tau>0$, the Nadaraya-Watson counterpart to the $k$-NN measure $\hat \mu_{n,k,x} $  (defined in the introduction) would be the following measure 	
\begin{align*}
\hat  \nu_{n,\tau ,x} ( A ) :  = \frac{   \sum_{i=1}^n \ind _ A(Y_i) \ind _{  B(x,  \tau  )  }(  X_i) }{  \sum_{i=1}^n \ind _{ B (x,  \tau )  }(  X_i) }  ,\qquad   (A \in \mathcal S).
\end{align*}
In the above formulation, the bandwidth $\tau$ no longer depends on the sample as opposed to the $k$-NN formulation. This makes the analysis easier even though one should examine the behaviour of the random denominator $ \sum_{i=1}^n \ind _{ B (x,  \tau )  }(  X_i) $.
In fact a similar result to Theorem \ref{weakcv:kNN} might be obtained but the asymptotic variance shall be different due to the fixed bandwidth $ \tau$ and the random denominator. A quick inspection reveals that the asymptotic covariance of the normalized process $  \sqrt {n\tau ^d}  (\hat  \nu_{n, \tau,x} ( g )   -   \mu_{x} ( g ) )$ is given by
$ \cov_x(g_1,g_2) / (V_d  f_X(x) )$. 
It differs significantly from the one  of the $k$-NN process as the Nadaraya-Watson variance is badly affected by low density region. {We do not claim the $k$-NN estimator to be superior to the Nadaraya-Watson estimator in terms of prediction performance. Our point is only that parametrizing local averaging using the number of neighbors $k$ instead of the size of the balls $\tau $ makes the inference easier as the variance term does not depend on the covariates density.
The final performances of the two regression estimators actually depend on the respective choice of $k$ and $ \tau $, which, if conduced similarly, are likely to produce similar results.
In fact, the $k$-NN measure might be seen as a Nadaraya-Watson measure where the bandwidth  $\tau$ is set equal to $ ( k  / (n  V_d f_X(x) ) )^{1/d} $. This connection between the two approaches is  observed in Lemma \ref{prop:tau_x} where the (random) $k$-NN radius is shown to be equivalent to $ ( k  / (n  V_d f_X(x) ) )^{1/d}$ with large probability. 
}

%
%

\textbf{Discussion of the entropy condition.}
{Condition \eqref{def:brack_cond} bears resemblance with the usual bracketing entropy condition useful to derive the weak convergence of the standard empirical process, expressed for instance in  Theorem \ref{vdvw_weakcv}. One difference is that the measure involved in \eqref{def:brack_cond} is the conditional measure $\mu_x$ instead of the common underlying probability measure (denoted by $P$ in the aforementioned theorem).  Another difference is that our integrability condition must hold uniformly in $Q$ belonging to the set of measures $\{\mu_{x,u} \, :\, |u|\leq \delta\}$ neighborhing the measure $ \mu_x$. {This might seem restrictive at first glance. However, for many classes of interest, e.g., class of cells $\mathcal I$ as well as the class of Euclidean balls, both introduced in Section \ref{sec:def}  (see also Chapter 2.7 in \cite{wellner1996} for more examples), upper bounds on the associated bracketing numbers are  independent from the underlying probability measure. As a consequence,  \eqref{def:brack_cond} can be obtained without involving additional assumptions on the set of measures $\{\mu_{x,u} \, :\, |u|\leq \delta\}$.}
}
Another way to obtain \eqref{def:brack_cond} is given in the next proposition. It requires the existence of a dominating measure for which a standard bracketing entropy condition is satisfied.  The proof of the next proposition is deferred to the end of the Appendix, Section \ref{app:prop_suffisient_cond}.

\begin{proposition}\label{prop:suffisient_cond}
Suppose that $ z\mapsto \mu_{z} (G^2)  $ is continuous at $x$ for  an envelope $G$ of $\mathcal G$ and that there is $\beta \geq 1$ and a probability measure $ \mu$ on $(S,\mathcal S)$  such that for all  $z\in B(x,\delta) $ and $A\in \mathcal S$,    $\mu_z(A) \leq \beta \mu(A)$. Then,  $J( \mathcal G,  L_2(\mu  )  ,  \delta) < \infty$ implies \eqref{def:brack_cond}.
\end{proposition}

We conclude by providing two examples of distribution $P$ for which  the dominating measure property holds true. First, suppose that  $(Y,X)$ has a joint density $f_{Y,X}$ and assume that it is bounded and bounded away from $0$ (which implies that $X$ has compact support). Then, it is easily shown that $\mu_z(A) \leq  \beta\lambda (A)$ for some $\beta$.
A second example is the regression model in which $Y - \mathbb E [Y|X] $ has density $f$ with respect to Lebesgue measure and $| \mathbb E [Y|X]| \ind_{X\in B(x,\delta) } \leq M $ almost surely (e.g., when the regression function is continuous). Then the distribution $\mu_x$ admits a density $y\mapsto f ( y - \mathbb E [Y|X = x] )$ bounded by $h(y) = \sup_{m \in [-M,M]} f(y - m) $ which is integrable in most standard cases (Gaussian, exponential, polynomial, compactly supported). Hence, one may choose $\mu $ as the probability distribution having density  proportional to $h$.

\subsection{Bounding the bias}\label{subsec:bounding_bias}

To obtain an upper bound on the bias term, defined in the introduction, we rely on two different regularity conditions. One first possibility occurs when the function $z\mapsto \mu_z ( g  )$ satisfies a Lipschitz condition at  point $x$, uniformly in $g\in \mathcal G$, i.e., there is $\tau >0$ and $L_{1,x}>0$, such that
\begin{align}\label{first_order}
   \sup_{g\in \mathcal G}  | ( \mu_{x }  - \mu_{z}  ) ( g  ) | \leq  L_{1,x}  \|x - z\| ,\qquad \forall z \in B (x,\tau) .
   \end{align}
The strength of the previous condition actually depends on $\mathcal G$. When the collection of functions $\mathcal G$ is the set of measurable functions valued in $[-1,1]$ the previous assumption becomes Lipschitz continuity in total variation norm. When $g$ is made of Lipschitz functions, it relates to Lipschitz continuity in $1$-Wasserstein distance. As stated formally in the subsequent theorem, the previous condition ensures that the bias term has order $  (k/n)^{1/d}  $. {Interestingly, better convergence rates, in  $ (k/n)^{2 /d} $, might be obtained under a stronger regularity condition by leveraging that $x\mapsto \ind_{B(0,1)}(x)$ is even symmetric. In the regression framework, this type of analysis is proposed for instance in Section 14.3 in \cite{biau2015lectures}. In our context, due to the presence of processes indexed by $\mathcal G$, the approach taken in the proof needs to be refined compared to the previous reference. We shall need that for each $g$, $ x\mapsto  \mu_{x }(g)$ is differentiable around $x$ and that its gradient satisfies
\begin{align}\label{second_order}
     \sup_{g\in \mathcal G}  \|  \nabla _x \mu_{x } (g)  -  \nabla_x \mu_{z}   ( g  ) \| \leq L_{2,x}   \|x - z\|, \qquad \forall z \in B (x,\tau),
   \end{align}
   for some $\tau >0$ and $L_{2,x}>0$. We now give a formal statement about the previous two  cases associated with first and second order regularity assumptions, respectively, \eqref{first_order} and \eqref{second_order}. 
   
\begin{theorem}\label{bias_bound}
In Theorem \ref{weakcv:kNN}, if \eqref{first_order} is fulfilled, then 
$$ \sup_{g\in \mathcal G} \left|  k ^{-1} \sum_{i \in  N_{n,k}(x)}   (\mu_{X_i} (g)   - \mu_x (g) ) \right|    = O_P( (k/n)^{1/d} )  + o_P(k^{-1/2} )  .$$
If moreover, \eqref{second_order} is fulfilled and $f_X$ is $L_{f,x}$-Lipschitz around $x$, we have
$$  \sup_{g\in \mathcal G}\left|  k ^{-1} \sum_{i \in  N_{n,k}(x)}   (\mu_{X_i} (g)   - \mu_x (g) ) \right|   = O_P( (k/n)^{2/d} ) + o_P(k^{-1/2} ) . $$

\end{theorem}

Note that the previous might be used to characterize the weak limit of the $k$-NN measure $ \hat \mu _{n,k,x}$ under a restrictive condition on the number of neighbors. First, if \eqref{first_order} holds true and $ k^{ (d+2)/2  } /n\to 0$, then
$
\{  k ^{1/2}  ( \hat \mu _{n,k,x}  - \mu _x) (g ) \} _{g\in \mathcal G} $ converges weakly to the Gaussian process described in Theorem \ref{weakcv:kNN}. Second, under \eqref{second_order} and  the condition that $ k^{ (d +4)/4  } /n\to 0$, then 
$
\{  k ^{1/2}  ( \hat \mu _{n,k,x}  - \mu _x) (g ) \} _{g\in \mathcal G} $ 
converges weakly to the same Gaussian process. This improves upon the previous case where $x\mapsto \mu_{x }$ is assumed to be Lipschitz only.
Note that the previous weak  convergence property is obtained at the price of choosing $k$ smaller than its optimal value so that the variance term is greater than the bias term. 

}

\subsection{Applications}

{
In this section, two illustrative applications of Theorem \ref{weakcv:kNN} are considered. We start by giving a Donsker-type theorem for the $k$-NN estimator of the conditional cumulative distribution function. We thereafter introduce a certain local-linear estimator based on the $k$-NN measure and we discuss ways to estimate its variance.

\textbf{$k$-NN  conditional cumulative distribution function.} Let $F_x$ denote the conditional cumulative distribution function of $Y$ given $X=x$ and 
define
$$ \hat F_{n,k,x} (y)  =  \hat \mu_{n,k,x} (\ind_{(-\infty , y ]}).$$ 
To apply previous results, the class $\mathcal G$ is taken equal to the class of cells $\mathcal I$, introduced in Section \ref{sec:def}, which satisfies  \eqref{def:brack_cond} because of the upper bound on the associated bracketing numbers that is recalled in Section \ref{sec:def}. Since the envelope function is $1$ for this class, \eqref{def:ui_local} is easily satisfied.
To ensure that the continuity conditions of Theorem \ref{weakcv:kNN} are fulfilled, it is enough to assume that $F_z (y )$ is continuous at $z=x$ for each $y$.
We are  now in position to state the following corollary of Theorem \ref{weakcv:kNN}.

\begin{corollary}\label{cor:weak}
Suppose that (A) is fulfilled. Let $x\in \mathbb R^d$ and assume that $f_X$ is 
continuous at $x$ and $f_X(x) >0$. Suppose that for each $y\in \mathbb R$, $z\mapsto F_z(y)$ is continuous at $x$.
If, in addition, $k:=k_n$ satisfies $   k \to \infty$ and $k/n \to 0 $, it holds that
\begin{align*}
\left\{  k ^{-1/2} \sum_{i \in  N_{n,k}(x)}   (  \ind_{ Y_i\leq  y}    -\mu_{X_i} ( \ind_{(-\infty, y] } )   )\right\}_{y\in \mathbb R}  \text{  converges weakly in $\ell^{\infty}(\mathbb R)$ to a Brownian bridge}
\end{align*}
with covariance $(y_1,y_2)\mapsto  F _{x}( y_1 \wedge y_2) - F _{x}(y_1) F _{x}(y_2)$.
\end{corollary}

The covariance function given above allows for easy inference as it might be estimated by $\hat F_{n,k,x} ( y_1 \wedge y_2) -  \hat F_{n,k,x} (y_1) \hat F_{n,k,x}(y_2)$. In contrast, the Nadaraya-Watson version has the following limiting covariance 
$ ( F _{x}( y_1 \wedge y_2) - F _{x}(y_1) F _{x}(y_2) ) (V_K / f_X (x)) $
which would require an additional estimate of $f_X(x)$ to approximate the variance. 
From Corollary \ref{cor:weak}, several results might be obtained using the delta-method as in \citep[Section 3.9.4]{wellner1996}. This includes, for instance, the weak convergence of the the $k$-NN conditional quantile estimate or the $k$-NN conditional copula estimate.


\textbf{Local linear nearest neighbor.}
We now consider the ordinary least-squares problem localizing with the $k$-NN measure. These types of estimators are similar to well-known {local linear estimators \citep{fan1996} which  are traditionally defined using Nadaraya-Watson weights.} Introduce the regression function $h(x) = \mathbb E [ Y| X=x] $ and  assume that $h$ is continuously differentiable at point $x$. We consider the loss function $\ell_{ \alpha, \beta}  (Y, X ) = (Y- \alpha -\beta ^T (X-x) ) ^2$ which, when associated with the $k$-NN measure (with respect to $ x$), gives the following empirical risk minimizer
\begin{align*}
(\hat \alpha_{n,k,x}   , \hat \beta_{n,k,x}  ) \in \argmin _ {\alpha\in \mathbb R,\, \beta\in \mathbb R^  d }    \hat \mu _{n,k,x} ( \ell _{\alpha, \beta} ) . 
\end{align*}
While a central limit theorem might be obtained as in the previous example, here we prefer, for the sake of clarity, to focus on variance estimation. As detailed in \citep{fan1996}, $ (\hat \alpha_{n,k,x}   , \hat \beta_{n,k,x}  )$ is an estimate of  $(h(x) , \nabla_x h(x) )$. The plug-in principle for variance estimation described previously, leads to the following variance estimate
$$  \text{var}  (\hat \alpha_{n,k,x}   , \hat \beta_{n,k,x}  )   \simeq  \sigma^2(x)  G_x ^+,  $$ 
where $G_x =  k  \hat \mu _{n,k,x}\{ aa^T \}$,  $a (X) = (1, ( X-x)^T)^T$, is the corresponding Gram matrix and $\sigma (x) $ is the conditional variance of $ Y  $ given $X = x$, which here is supposed to be known for simplicity (we refer to \cite{devroye2018nearest} for residual variance estimation). Note that the obtained formula is the same as in classic ordinary least-squares regression, except that only the $k$-NN points are involved and $X_i - x$ is used instead of $X_i$. Therefore, if one would seek to make inference on the gradient's magnitude, it would suffice to apply well-documented significance testing approaches taking care of the two changes just mentioned.
%

}
%
%

\section{Non-asymptotic bound under uniform entropy}\label{sec:uniform_bound}

\subsection{Statement of the result}\label{subsec:uniform_bound}


We now state our second main result, a non-asymptotic bound on the error associated with the nearest neighbor measure $\hat \mu _{n,k,x}$ estimating $\mu_x$. The proposed bound on the $k$-NN process  holds uniformly over $g$ lying in a VC class and over $x\in S_X$. Because of the last point, the assumptions now apply to the entire set $S_X$ in contrast with the previous section. More specifically, we require that there is $c>0$ and $T>0$ such that 
\begin{align}\label{cond:reg1}
&\lambda (S_X \cap  B(x, \tau ) ) \geq c \lambda   ( B(x, \tau )) , \qquad \forall \tau \in (0,T] , \, \forall x\in S_X,
\end{align}
and there is $0 < b_X\leq U_X <+\infty$ such that
\begin{align}\label{cond:reg2}
& b_X\leq f_{X}(x) \leq U_X , \qquad \forall x \in S_X .
\end{align}
The first above condition is meaningful for small $\tau>0$ and hence deals with the boundary of the set $S_X$ while the second means that the measure $P_X$ charges uniformly the set $S_X$ (no region with no point).

The proof of the next theorem is given in Section \ref{section:concistency=nn} using two lemmas, namely Lemma \ref{prop:tau} and Lemma \ref{boundW}, that are proved in the Appendix.


\begin{theorem}\label{consistenc:kNN}
{Suppose that (A) is fulfilled and that \eqref{cond:reg1} and \eqref{cond:reg2} hold true. Suppose that the function class $\mathcal G$ is VC with parameter $(v,A)$ and constant envelope $1$.  Then, for all $n\geq 1$, $\delta \in (0,1)$ and $1\leq k \leq n$ such that 
\begin{align*}
8   d  \log(24 n/\delta ) \leq k  \leq  n \left\{  \left(\frac{ 2 }{ \sigma_{\mathcal G} ^2   \kappa_X }  \right) \wedge  \left( T ^d   b_X c V_d /2  \right)\right\} ,
\end{align*}
with $\kappa _X =  {U_X} / {cb_X}$, $\sigma_{\mathcal G} ^2  = \sup_{g\in \mathcal G,\, x\in S_X} \Var_{\mu_x } (g)  $,
it holds, with probability at least $1-\delta$:
\begin{align*}
&\sup_{g \in \mathcal G, \, x\in S_{X}} \left| k ^{-1} \sum_{i \in  N_{n,k}(x)}   (  g(Y_i )   -\mu_{X_i} (g)   )  \right| \\
& \leq 
 2K_1 \sqrt{   \frac{ ( d+1 +v )  \sigma_{\mathcal G} ^2 \kappa_X }{k} \log \left(  52 e^2  \frac{A n } { \sigma _{\mathcal G} \delta}    \right)  } 
+   4K_2 \frac{ ( d+1 +v ) }{k}  \log \left( 52 e^2  \frac{A n } { \sigma _{\mathcal G} \delta} \right)  
\end{align*}
with $K_1  $ and  $K_2 $ are the explicit constants from Theorem \ref{vc_bound}.
}
\end{theorem}

\subsection{Discussion and remarks}

\textbf{Rates of convergence.}
The bound given in Theorem \ref{consistenc:kNN} is concerned with the variance component and the  two terms together form a typical Bernstein bound. Following the result of Section \ref{subsec:bounding_bias}, namely Theorem \ref{bias_bound},  we recover optimal rates for respectively differentiable and twice-differentiable  functions. Indeed, if
$$\forall (x,  z)  \in S_{X} \times S_{X} ,\qquad  \sup_{g\in \mathcal G} | \{ \mu_x   - \mu_{z} \}(g)  | \leq  L_1 \|x-z\|,$$
 then the bias is of order $(k/n)^{1/d}$ and hence, as for the $k$-NN regression estimator \citep{gyorfi1981rate,jiang2019non}, optimal rates of convergence for Lipschitz functions  \citep{stone1982optimal} are achieved when selecting $  k  =   n^{2 / (d+2) }   $ as it gives a uniform error of order $ n ^{ -  1 / (d+2) }  $. Under stronger regularity conditions, in particular, if $x\mapsto \mu_x(g)$ is differentiable for each $g\in \mathcal G$,  and if
 $$\forall (x,  z)  \in S_{X} \times S_{X} ,\qquad  \sup_{g\in \mathcal G} |  \nabla _x \mu_x (g)  -  \nabla _x \mu_{z} (g)  | \leq  L_1 \|x-z\|,
$$ 
then the bias term is of order $(k/n)^{2/d}$. This can be deduced from the proof of the second assertion in Theorem \ref{bias_bound} using tools from the proof of  Theorem \ref{consistenc:kNN}. In this case, the optimal rate of convergence is $ n ^{ -  2 / (d+4) }$ when choosing $k = n^{4/ (d+4)} $.

{\textbf{Constants.} The previous non-asymptotic bound depends explicitly on some characteristic of the problems (e.g., density conditioning number $\kappa_X$ or the VC class parameters) as well as  known absolute constants $( C,K_1,K_2 )$. Even though the constants are not meant to be optimal we emphasize that in some scenarios one may be able to approximate (if not compute) the given upper-bound.}

\textbf{Sketch of the proof.}
The approach employed in the proof of Theorem \ref{consistenc:kNN} relies on controlling the complexity of the nearest neighbors selection mechanism. The basic idea is to embed the nearest neighbors selector around $x$, namely $\ind_ {B(x, \hat \tau_{n,k,x}  ) } $, in a larger class of kernel functions, $\ind_ {B(x,\tau )} $ where $\tau $ lies between $0$ and $(k/n)^{1/d}$ (up to some factor specified in the proof). Since this class of kernel functions has a similar VC dimension as the class of balls studied in \cite{wenocur1981some}, we are able to use Theorem \ref{vc_bound}.

\textbf{From uniform regression to uniform bounds on processes.}
A related result, where the bound obtained is also uniform in $x$, has been obtained in Theorem 1 of  \cite{jiang2019non}. The difference is that Theorem \ref{consistenc:kNN} above is valid for the $k$-NN measure (involving a function class $\mathcal G$) whereas the result in  \cite{jiang2019non} is obtained in the more specific case of (Lipschitz) regression. Comparing the two approaches, the most significant difference turns out to be in the variance term (which is examined in Theorem \ref{consistenc:kNN}). More precisely, there is a new term, $(d\vee v)$, that shows-up in our bound compared to Jiang's bound where only $d$ is needed. The factor $v$ accounts specifically for the complexity of the class $\mathcal G$ while the factor $d$ comes from the uniform requirement with respect to $x$. If $\mathcal G$ would be made of $1$ element, then only $d$ would be needed in the bound as in  \cite{jiang2019non}. In addition, in Theorem \ref{consistenc:kNN}, the scaling condition between $k$, $n$, $d$ and $\delta$, needed in \cite{jiang2019non}, Theorem 1, has been improved from $d\log(1/\delta)^2\log(n) \leq k $ to $d\log(n/\delta) \leq n $ (taking all constants equal to $1$ for clarity).

{\textbf{Low density regions.}
Theorem \ref{consistenc:kNN} relies on several good  properties of the covariate density $f_X$ which in particular should not approach $0$. 
This assumption is key because it allows to obtain a rate of convergence for the $k$-NN radius $\hat \tau_{k,n,x}$ of order $(k/n)^{1/d}$. This was also the case in Theorem \ref{weakcv:kNN} because $f_X(x)$ ($x$ is fixed) was assumed to be positive. When this assumption is no longer valid, it might be difficult (perhaps not possible) to obtain the same rate as the ones described in the previous results. However,  it might be interesting to examine errors that are averaged over the covariates domain such as the $L_2$-risk. Note that several related results have already been obtained concerning the analysis of the $k$-NN classification rule when covariates have low density regions. Excess risk bounds are examined in \cite{gadat2016classification} and \cite{cannings2020local} under different margins conditions.}

\section{Proofs of the main results}\label{section:proofs}

We now give the proofs of Theorem \ref{weakcv:kNN}, \ref{bias_bound} and \ref{consistenc:kNN} by relying on several technical lemmas, namely Lemma \ref{prop:tau_x}, \ref{lemma:brecketing_cond}, \ref{prop:tau} and \ref{boundW}, which proofs are given in the Appendix.

\subsection{Proof of Theorem \ref{weakcv:kNN}}\label{section:proofs_weak}

The proof is made up of two steps, as detailed below.

 \textbf{Step (i): approximating the $k$-NN radius.}
  We set
\begin{align*}
 \tau_{n,k,x}  = \left(  \frac k n \frac 1 {f_X(x) V_d} \right) ^{1/d}   .
\end{align*}
Next we claim that $\hat \tau_{n,k,x}  $ is equivalent to $  \tau _{n,k,x} $, in probability. The proof can be found in the Appendix.

\begin{lemma}\label{prop:tau_x}
Let $x\in \mathbb R^d$ and suppose that $f_X$ is continuous at $x$ and $f_X(x) >0$, if $k/n\to 0$ and $ k \to\infty   $, we have
$  (\hat \tau_{n,k,x}  /   \tau _{n,k,x}  )^d \to  1$, in probability,  as $n\to \infty$.
\end{lemma}

 \textbf{Step (ii): examining the bracketing entropy.} The function class 
\begin{align*}
\mathcal F_{n,k,x} = 
&\left\{(y ,z) \mapsto \sqrt { (n/k) }  f(y,z)     \ind _{ B(x, u^{1/d} \tau_{n,k,x})  }(z) \, :\, f\in \mathcal F ,\, u \in  [ 1/2 ,  3/2 ] \right\} 
\end{align*}
will play an important role in the proof. Next we show that the bracketing entropy of $ \mathcal F_{n,k,x} $ is small enough as soon as $\mathcal F$ satisfies the following: for some $\delta>0$ and for any positive sequence $(\delta_n)_{n\geq 1}$ going to $0$, it holds that
	\begin{equation}\label{cond_bracket2}
\limsup _{n\geq 1}   \sup_{ |u|\leq  \delta   } J( \mathcal F,  L_2(\tilde \mu_{x,u} )  ,  \delta_n)  =  0, 
	\end{equation}
 where  $J$ is the bracketing integral introduced in Section \ref{sec:def} and $\tilde \mu_{x,u}$ is similar to $ \mu_{x,u}$ but is extended to space $\mathcal S$. For any $u > 0$, 
 $$ \tilde \mu_{x, u} ( A \times B) = \frac{ E (  \mu_X (A)  \ind_ B(X) \ind _{  B(x , u^{1/d} )  }  (X) )  } {  E (   \ind _{  B(x , u^{1/d} )  }  (X) )   }  ,\qquad A \in \mathcal  S, B\in \mathcal B ( \mathbb R^d).$$  The proof of the next result is given in the Appendix.

\begin{lemma}\label{lemma:brecketing_cond}
Let $x\in \mathbb R^d$ and suppose that $f_X$ is continuous at $x$ and $f_X(x) >0$. Assume that there is $\delta >0 $ for which \eqref{cond_bracket2} holds true  with envelope $F$ such that $ z\mapsto \mu_{z} (F^2(\cdot, z))  $ is continuous at $x$. Then, if $k/n\to 0$, for any positive sequence $(\delta_n)_{n\geq 1}$ going to $0$, 
\begin{align*}
\lim_{n\to \infty}  \int_0^{\delta_n}  \sqrt { \log \mathcal N _{[\,]}\left( \mathcal F_{n,k,x}  , L_2(P) , \epsilon \| F _{n,k,x}\| _{L_2(P) }\right) }   \, \diff \epsilon = 0,
\end{align*}
where $F _{n,k,x} = \sqrt { (n/k) } F\ind _{ B (x, (3/2)^{1/d} \tau_{n,k,x}) } $ is an envelope for $\mathcal F _{n,k,x}$.
\end{lemma}

Define the residual function
$\epsilon_g (y,z) =  g(y )   - \mu_{z} (g ) $ and the associated function class
 \begin{align*}
 { \mathcal E} _{n,k,x} = 
&\left\{(y ,z) \mapsto \sqrt { (n/k) } \epsilon_ g(y,z)   \ind _{ B(x, u^{1/d} \tau_{n,k,x})  }(z) \, :\, g\in \mathcal G ,\, u \in  [ 1/2 ,  3/2 ] \right\} 
\end{align*}
with envelope  $E _{n,k,x} = \sqrt { (n/k) } (2G)\ind _{ B (x, (3/2)^{1/d} \tau_{n,k,x}) } $. 
To conclude this step, we apply the previous lemma to obtain that
\begin{align}\label{cond_bracket_epsilon} 
\lim_{n\to \infty}  \int_0^{\delta_n}  \sqrt { \log \mathcal N _{[\,]}\left( \mathcal E_{n,k,x}  , L_2(P) , \epsilon \| E _{n,k,x}\| _{L_2(P) }\right) }   \, \diff \epsilon = 0.
\end{align}
 This shall be used when applying Theorem \ref{vdvw_weakcv} in the next paragraph. To apply Lemma \ref{lemma:brecketing_cond}, we need to show that the  space  containing residual functions 
$$ \mathcal E = \{ (y,z) \mapsto  \epsilon_g (y,z) =  g(y )   - \mu_{z} (g ) \, : \, g\in \mathcal G\}$$
 satisfies \eqref{cond_bracket2} with envelope function $2G$. Let $\epsilon \in (0,1) $ and $([\underline g_k,\overline g_k])_{k=1,\ldots, M_\epsilon}  $ be a collection of $(\epsilon \|G\|_{ L_2 (\mu_{n,k,x}  ) } , L_2 ( \mu_{n,k,x} ) )$-brackets covering $\mathcal G$. We can assume that $\overline g_k\leq G$ and $ \underline g_k\geq -G$. If this would not be the case, one would take $ \overline g_k \wedge G$ in place of $\overline g_k$ and similarly for $\underline g_k$.
Define $\underline  \epsilon_k(y,z) = \underline g_k(y) - \mu _z (\overline g_k)$ and $\overline \epsilon_k(y,z) =   \overline g_k(y) - \mu _z (\underline g_k) $ and note that $[\underline  \epsilon_k  , \overline  \epsilon_k ]$, $k=1,\ldots, M_\epsilon $,  covers $\mathcal E $. Using Minkowski's inequality and then Jensen's inequality, we find 
\begin{align*}
\| \underline  \epsilon_k  - \overline  \epsilon_k \|_{L_2 (\tilde \mu_{n,k,x})} &\leq \|  \underline g_k  - \overline g_k \|_{L_2 (\mu_{n,k,x})} + \|  \mu _z (\underline g_k  - \overline g_k) \|_{L_2 (\tilde \mu_{n,k,x})}\\
&\leq 2 \|  \underline g_k  - \overline g_k \|_{L_2 (\mu_{n,k,x})}\\
&  \leq 2\epsilon  \|G\|_{ L_2 ( \mu_{n,k,x} )} .
\end{align*}
It follows that $([\underline \epsilon_k,\overline \epsilon_k])_{k=1,\ldots, M_\epsilon}  $ is a collection of $(2\epsilon \|G\|_{ L_2 (\mu_{n,k,x}  ) } , L_2 ( \mu_{n,k,x} ) )$-brackets covering $\mathcal E$ and satisfying $-2G\leq  \underline \epsilon_k\leq \overline \epsilon_k\leq 2G$. As a consequence, $\mathcal E$ satisfies \eqref{cond_bracket2} with envelope $2G$.

 \textbf{Step (iii): End of the proof.}  Define the process, for any function $f$ and $u \in [1/2 ,3/2]$
\begin{align*}
\hat Z_{n,k,x} (f,u) = k^{-1/2} \sum_{i=1}^n \{  f(Y_i,X_i)  \mathds{1} _{  B  (x, u^{1/d}  \tau_{n,k,x} )  }(  X_i)  - P ( f  \mathds{1} _{  B  (x, u^{1/d}  \tau_{n,k,x} )  } ) \}.
\end{align*}
Because of the definition of $\hat \tau_{n,k,x} $ in Section \ref{subsec:weak_cv_rk}, we have that
 $$k ^{-1/2} \sum_{i \in  N_{n,k}(x)}   (  g(Y_i )   -\mu_{X_i} (g)   ) =  k^{-1/2} \sum_{i=1}^n  (  g(Y_i )   -\mu_{X_i} (g)   ) \ind _{  B(x,  \hat \tau_{n,k,x}   )  }(  X_i) ,$$ 
but since $P ( \epsilon_g \mathds{1} _{  B  (x, u^{1/d}  \tau_{n,k,x} )  } )  = 0$, it follows that
  $$k ^{-1/2} \sum_{i \in  N_{n,k}(x)}   (  g(Y_i )   -\mu_{X_i} (g)   ) =  \hat Z_{n,k,x} (\epsilon _ g,\hat u_{n,k,x} ),$$
 where $\hat u_{n,k,x}  : = (\hat \tau_{n,k,x}   / \tau_{n,k,x})^d $. The fact that $\hat u_{n,k,x} $ converges to $1$ (see Lemma \ref{prop:tau_x}), suggests the use of the following decomposition
\begin{align*}
& k ^{-1/2} \sum_{i \in  N_{n,k}(x)}   (  g(Y_i )   -\mu_{X_i} (g)   ) = \hat Z_{n,k,x} (\epsilon_ g,1)  + \{\hat Z_{n,k,x} (\epsilon_ g,\hat u_{n,k,x} )  - \hat Z_{n,k,x} (\epsilon_ g,1)\}.
\end{align*}
A relevant property is that $\hat Z_{n,k,x} (\epsilon_ g,u)$ converges weakly to a tight Gaussian process with covariance function
$$ ((g_1,u_1 ) , (g_2,u_2) ) \mapsto (u_1\wedge u_2) \{ \mu_x (g_1g_2)  -  \mu_x (g_1) \mu_x(g_2) \} .$$ 
Admitting this property for now;
we have  \citep[eq. (2.1.8)]{wellner1996} that 
$$\sup_{ |u-1|\leq \delta_n } | \hat Z_{n,k,x} (\epsilon_ g,u) - \hat Z_{n,k,x} (\epsilon_ g,1) | \to 0$$
 in probability for all $\delta_n \to 0$. But since by Lemma \ref{prop:tau_x}, $|\hat u_{n,k,x}  - 1| \to 0$ in probability, it follows that 
$$\hat Z_{n,k,x} (\epsilon_ g,\hat u_{n,k,x} ) - \hat Z_{n,k,x} (\epsilon_ g,1) = o_{P}(1) .$$
The fact that $\hat Z_{n,k,x} (\epsilon_ g,1)$ converges to a Gaussian process is a consequence of the weak convergence property of $\hat Z_{n,k,x}$, which now is the last thing we need to establish to conclude the proof. 

The proof that $\{ \hat Z_{n,k,x} (\epsilon_ g,u)\, : \,  g\in \mathcal G,\, u\in [1/2,3/2]\} $ converges weakly (with the specified limiting covariance) follows from an application of the functional central limit theorem  stated in Section \ref{sec:def} as Theorem \ref{vdvw_weakcv}. The class of interest is 
${ \mathcal E} _{n,k,x} $ as introduced in step (ii). The index semimetric space is $T =  {\mathcal G}\times [1/2,3/2]$ endowed with the distance 
$$ \rho : (( g_1,u_1),( g_2,u_2) ) \mapsto   \rho_{\delta} (g_1,g_2) + |u_1-u_2|,$$
recalling that $\rho_\delta ( g_1,g_2) = \sup_{ |u| \leq \delta}   \|  g_1 - g_2 \|_{L_2(  \mu_{x, u})} $. 
The condition on the bracketing numbers of ${ \mathcal E} _{n,k,x} $ needed in Theorem \ref{vdvw_weakcv}, is exactly  \eqref{cond_bracket_epsilon} which is established in step (ii).
 To apply Theorem \ref{vdvw_weakcv}, it remains to show that
\begin{align}
 \nonumber \limsup _{n\geq 1} P(E _{n,k,x}^2)  <\infty &, \\
 \nonumber P (E _{n,k,x}^2 \mathds 1_{E _{n,k,x} >\eta \sqrt n } ) \to 0 \quad \text{as } n\to \infty&,\\
  \label{equi_sample_path}
\sup_{\rho ((g_1,u_1) ,(g_2,u_2)  ) \leq \delta_n} P ( \epsilon_{n,(g_1,u_1)} - \epsilon_{n,(g_2,u_2)} )^2 \to 0&, \qquad \text{as } \delta_n \to 0,
\end{align}
where $\epsilon_{n,(g,u)}$ is the member of $ \mathcal E_{n,k,x} $ associated with $g$ and $u$. The first and second conditions are easy to obtain. The second, also called the Lindeberg condition, follows from the uniform integrability of the family of measures $\mu_z$ when $z$ lies in $B(x,\delta)$. Indeed, let $\eta>0$ and $\epsilon >0$ be arbitrary and set $u_n = (3/2) \tau_{n,k,x}^d$.  Choose $M>0$ such that $\mu_{x, u_n } \{  G^2 \mathds 1_{ G > M} \} \leq \epsilon$ for all $n\geq  1 $ (this is in virtue of the uniform integrability assumption).  Note that for $k$ large enough and $k/n$ small enough, 
\begin{align*}
P (E _{n,k,x}^2 \mathds 1_{E _{n,k,x} >\eta \sqrt n } ) &\leq  4 (n/k) P (G^2 \ind _{ B (x, (3/2)^{1/d} \tau_{n,k,x}) } \mathds 1_{ 2G >\eta \sqrt k } ) \\
&\leq  4 (n/k) P (G^2 \ind _{ B (x, (3/2)^{1/d} \tau_{n,k,x}) } \mathds 1_{ G > M} )\\
&=  4 (n/k) \mu_{x,u_n }  \{  G^2 \mathds 1_{ G > M} \} P (   \ind _{ B (x,  u_n^{1/d} ) }  )\\
&\leq (  4\epsilon )  (n/k) P (  \ind _{ B (x,u_n ^{1/d}) } )\\
&=(4\epsilon) (n/k)  2 f_X(x)  (3/2) \tau_{n,k,x}^d V_d \\
&=12\epsilon .
\end{align*}
Using the previous algebra, the first condition is obtained from 
\begin{align*}
P E _{n,k,x}^2 &= 4 (n/k) P (G^2 \ind _{ B (x, u_n^{1/d}) } ( \ind _{G>M }  + \ind_{G\leq M} ) ) \\
& \leq 12 (\epsilon +  M^2  ).
\end{align*}
We now show that \eqref{equi_sample_path} is valid. Suppose that $t_1 < t_2$ and $g\in \mathcal G$, it holds, for $t_2$ small enough (so that $ \mu_z( G^2 ) f _{X} (z)\leq 2  \mu_x( G^2 ) f _{X} (x)$ whenever $\|z-x\|\leq t_2$),
\begin{align*}
 P (   \epsilon_g ^2 (  \ind _{B( x, t_2    )  }  -   \ind_{B( x, t_1  )  } )  ^2 )  &\leq   4 P (   G^2  ( \ind _{B( x, t_2   )  }  -   \ind_{B( x,t_1   )  } ) ) \\
 &  = 4 \int _{  t_1   \leq \|x- z \|\leq t_2    }     \mu_z( G^2 ) f _{X} (z) \, \diff z \\
 & \leq  8 \mu_x( G^2 ) f _{X} (x)   \int _{  t_1    \leq   \|x- z \|\leq t_2   }   \, \diff z \\
& \leq  8\mu_x( G^2 ) f _{X} (x)    (t_2^d - t_1^d) V_d .
\end{align*}
Taking $t_1 =  u_1^{1/d } \tau_{n,k,x}$ and $t_2 =  u_2^{1/d } \tau_{n,k,x}$, with $u_2>u_1$, we get, for $n$ large enough,
 \begin{align*}
 P (    \epsilon_g ^2 (  \ind _{B( x, u_2^{1/d } \tau_{n,k,x} )  }  -   \ind_{B( x, u_1^{1/d } \tau_{n,k,x} )  } )  ^2 )  & \leq  8 \mu_x( G^2 ) (k/n) (u_2-u_1).
\end{align*}
Consequently, for any $\delta _n \to 0$,  
$$ \sup_{ |u_1 - u_2 | \leq \delta_n } (n/k)  P\left\{     \epsilon_g  ^2 (  \ind _{B( x, u_2^{1/d } \tau_{n,k,x} )  }  -   \ind_{B( x, u_1^{1/d } \tau_{n,k,x} )  } )  ^2 \right\}\to 0.$$
Suppose now that $ u \in [1/2, 3/2] $ and  $(g_1, g_2) \in \mathcal G\times \mathcal G$.  
For $n$ large enough,
\begin{align*}
 P \{ (  (g_1  -  g_2 ) \mathds 1_{B(x,  u^{1/d}  \tau _{n,k,x} )  }    )^2\}
 & \leq P ( (g_1  -  g_2 )^2 \mathds 1_{B(x,  (3/2)^{1/d}  \tau _{n,k,x} )  }   )\\
 & = \mu_{x,(3/2)  \tau _{n,k,x}  ^d } (g_1  -  g_2 )^2 P (\mathds 1_{B(x,  (3/2)^{1/d}  \tau _{n,k,x} )  } ) \\
& \leq    \rho _\delta ( g_1 , g_2 )^2P (\mathds 1_{B(x,  (3/2)^{1/d}  \tau _{n,k,x} )  } )   \\
& \leq 2 \rho _\delta ( g_1 , g_2 )^2 f_X(x)  (3/2) \tau_{n,k,x} ^d    V_d.
\end{align*}
Consequently, for any $\delta _n \to 0$, 
$$\sup_{   \rho _\delta ( g_1 , g_2 ) \leq \delta_n }  (n/k) P\{ (  (g_1  -  g_2 ) \mathds 1_{B(x,  u^{1/d}  \tau _{n,k,x} )  }    )^2  \} \to 0  ,$$
and, invoking Minkowski and then Jensen's inequality, we obtain
$$\sup_{   \rho _\delta ( g_1 , g_2 ) \leq \delta_n }  (n/k) P \{  ( (g_1 - \mu_X(g_1)   -  (g_2 - \mu_X(g_2) )) \mathds 1_{B(x,  u^{1/d}  \tau _{n,k,x} )  }  ) ^2   \}  \to 0  . $$
It follows that
\begin{align*}
&\sup_{   \rho _\delta ( g_1 , g_2 )   \leq \delta_n, \, |u_1- u _2 |\leq \delta_n }  
(n/k) P \left\{  \left(\epsilon_{g_1}  \mathds 1_{B(x,  u_1^{1/d}  \tau _{n,k,x} )  }-\epsilon_{g_2}   \mathds 1_{B(x,  u_2^{1/d}  \tau _{n,k,x} )  } \right)^2 \right\}     \to 0,
\end{align*}
and we have just obtained \eqref{equi_sample_path}. We can now conclude by computing the limiting covariance as follows. Let $ (u_1,u_2)  \in [1/2, 3/2]^2 $ and  $(g_1, g_2) \in \mathcal G\times \mathcal G$ such that $u_1<u_2 $. Define $c_X( g_1 ,g_2) =  \mu_X(g_1  g_2 ) - \mu_X(g_1) \mu_X(g_2)$. We have, for all positive function $f$,
\begin{align*}
E( ( g_1(Y) - \mu_X(g_1)) (g_2(Y) - \mu_X(g_2)) f(X) ) = E(  c_X(g_1,g_2) f(X) ).
\end{align*}
We then obtain
\begin{align*}
&(n/k) \cov\left( ( g_1 - \mu_X(g_1) ) \ind _{B(x,u_1^{1/d} \tau_{n,k,x} )  }, (g_2 - \mu_X(g_2) )\ind _{B(x,u_2^{1/d} \tau_{n,k,x} )  } \right ) \\
& = (n/k)P (c_X(g_1,g_2)  \ind _{B(x,u_1^{1/d} \tau_{n,k,x} )  }\ind _{B(x,u_2^{1/d} \tau_{n,k,x} )  } ) \\
&= (n/k) P (c_X(g_1,g_2)  \ind _{B(x,u_1^{1/d} \tau_{n,k,x} )  } ) .
\end{align*}
From the continuity of $z\mapsto c_z(g_1,g_2)  $ and $f_X$ at $x$ we find  that the last term converges to $u_1 c_{x} (g_1g_2)$ which corresponds to the stipulated covariance.
\qed

%

\subsection{Proof of Theorem \ref{bias_bound}}


{

Start by writing the following decomposition
\begin{align*}
& k ^{-1}    \sum_{i = 1 }^n  ( \mu_{X_i} (g)    - \mu_x (g) ) \ind _{  B(x, \hat \tau_{n,k,x}   )  } (  X_i)  \\
&=  
k^{-1/2} \hat Z_{n,k,x} (b_g, \hat u) +  \frac{ n}{k} P[  b_g  \ind _{  B(x, \hat \tau_{n,k,x} )  } ]
\end{align*}
where  $b_g (z) = \mu_{z} (g)    - \mu_x (g) $, $ \hat  u_{n,k,x} =  ( \hat \tau_{n,k,x} /   \tau_{n,k,x})^d$ and  $\hat Z_{n,k,x} $ is introduced in the  proof of Theorem \ref{weakcv:kNN}.
Following the proof of Theorem \ref{weakcv:kNN}, it is easy to show that
$ \{\hat Z_{n,k,x} ( b_g ,u) \, : \, g\in \mathcal G\}$  converges weakly to a Gaussian process. The only difference is that the functions involved are  $z\mapsto  \mu_{z} (g)    - \mu_x (g)$ instead of $(y,z) \mapsto  g(y) - \mu_z(g) $ in the proof of Theorem \ref{weakcv:kNN}. 
 The covariance function of the limiting Gaussian process is $0$ because the class of interest $b_g$ shrinks as $  \tau_{n,k,x} \to 0$. More precisely, we observe that, in virtue of the $L_{1,x}$-smoothness assumption,
$$\sup_{g\in\mathcal G,\, u\in (1/2,3/2)} (n/k)   P[ b_g ^2  \ind _{  B(x, u^{1/d}  \tau_{n,k,x} )  }  ]  \to 0 .$$  
As a consequence, it holds that $ \hat Z_{n,k,x} (b_g  , \hat u)  \to 0  $,  in probability. Therefore, we have shown that
$$ k ^{-1}    \sum_{i = 1 }^n  b_g (X_i) \ind _{  B(x, \hat \tau_{n,k,x}   )  } (  X_i) = o_P(k^{-1/2} ) + \frac{ n}{k} P[ b_g  \ind _{  B(x, \hat \tau_{n,k,x} )  } ].$$
It now remains to study the term $ (n/k) P[  b_g \ind _{  B(x, \hat \tau_{n,k,x} )  }]$ considering the two stipulated regularity assumptions. Under the $L_{1,x}$-smoothness assumption, we have
\begin{align*}
(n/k)  P[  b_g \ind _{  B(x, \hat \tau_{n,k,x} )  } ] &\leq L_{1,x}  \hat \tau_{n,k,x}  (n/k) P \ind _{  B(x, \hat \tau_{n,k,x} )  } \\
& \leq L_{1,x}  \hat \tau_{n,k,x}^{d+1}   (n/k)  \int_{B(0,1) }  f_X(x+   \hat \tau_{n,k,x}   u ) du  
\end{align*}
 which, in virtue of Lemma \ref{prop:tau_x}, is $ O_P((k/n)^{1/d}  )$.
 We have, by the $L_{2,x}$-smoothness assumption,
\begin{align*}
&P[ b_g \ind _{  B(x, \hat \tau_{n,k,x} )  }  ] \\
&\leq   \left\| \nabla_x  \mu_x (g) \right\| \left\| E ( X_1   - x  ) \ind _{  B(x,  \hat \tau_{n,k,x}   )  } (  X_1) \right\| + L_{2,x}   E [ \| X_1 - x  \|^2 \ind _{  B(x,  \hat \tau_{n,k,x}   )  }(  X_1)]  .
\end{align*}
 Concerning the left-hand side term,  we have, for any small enough $\tau>0$ (to use the $L_{f,x}$-Lipschitz property of $f_X$)
\begin{align*}
 \|  E [ ( X_1   - x  )  \ind _{  B(x,  \tau  )  }(  X_1)  ]\| &= \tau  \| \int \frac{(y - x)}{\tau}  \ind _{  B(x,  \tau  )  }(  y) f_X(y) dy\|\\
 & =  \tau ^{d+1}  \| \int  u   \ind _{  B(0,  1 )  }(u ) f_X(x+\tau u ) du\|\\
 &  = \tau ^{d+1}  \| \int  u   \ind _{  B(0,  1 )  }(u ) (f_X(x+\tau u )  - f_X(x )  ) du\|\\
 & \leq L_{f,x} \tau ^{d+2}   \int  \| u\|^2    \ind _{  B(0,  1 )  }(u )  du\\
 &\leq V_d L_{f,x} \tau ^{d+2}  .
\end{align*}
Concerning the right-hand side term,  we have, for any $\tau>0$,
\begin{align*}
   E [ \|  X_1   - x  \|^2  \ind _{  B(x,  \tau  )  }(  X_1)  ]  &= \tau^2   \int \|  \frac{(y - x)}{\tau}  \|^2 \ind _{  B(x,  \tau  )  }(  y) f_X(y) dy\\
 & =  \tau ^{d+2}   \int \|  u  \|^2  \ind _{  B(0,  1 )  }(u ) f_X(x+\tau u ) du.
 \end{align*}
As a consequence, in virtue of Lemma \ref{prop:tau_x}, we find
\begin{align*}
P[b_g \ind _{  B(x, \hat \tau_{n,k,x} )  } ]  = O ( \hat \tau_{n,k,x}^{d+2}  )  = O_p( (k/n) ^{1 +2/d}  ) .
\end{align*}
 The result follows multiplying by $(n/k)$.
 \qed
}

\subsection{Proof of Theorem \ref{consistenc:kNN}}\label{section:concistency=nn}

The first two steps consist in independent intermediate results. Their  proofs are given in the Appendix. They will be put together in the third and last step of the proof.

 \textbf{Step (i): Bounding $k$-NN radius.} Define 
\begin{align*}
&\overline{\tau}_{n,k}  = \left(\frac{ 2  k }{ n b_X c V_d}  \right)^{1/ d}.
\end{align*}
The following Lemma controls the size of the $k$-NN balls uniformly over all $x\in S_X$. In this way, it extends Lemma \ref{prop:tau_x} which focuses on a particular $x$. 
The proof is given in the Appendix.

\begin{lemma}\label{prop:tau}
Suppose that \eqref{cond:reg1} and \eqref{cond:reg2} hold true. Then, for all $n\geq 1$, $\delta \in (0,1)$ and $1\leq k\leq n$ such that $8 d \log(12n / \delta ) \leq k   \leq T ^d  n b_X c V_d /2 $
, it holds, with probability at least $1-\delta$:
\begin{align*}
 \sup _{x\in S_X} \hat  \tau_{n,k,x}   \leq \overline \tau_{n,k}.
\end{align*}
\end{lemma}

 \textbf{Step (ii):  Bounding a certain $k$-NN process.} The  quantity 
\begin{align*}
\hat Z_{n,k}=  \sup_{  g\in \mathcal G , \, x\in S_X , \, \tau \in (0,\overline \tau_{n,k}]  }  \left|  \sum_{i=1}^n     (g(Y_i)   - \mu_{X_i}  (g) )  \mathbb I _{  B (x,\tau)   }( X_i)  \right | .
\end{align*}
is important as it can be used to upper bound the $k$-NN process. To obtain an upper-bound on $\hat Z_{n,k}$, we first show that the underlying class of functions is VC and then apply Theorem \ref{vc_bound}. In this way is obtained the following lemma whose proof is given in the Appendix.

\begin{lemma}\label{boundW}
Suppose that \eqref{cond:reg1} and \eqref{cond:reg2} hold true and that the function class $\mathcal G$ is VC with parameter $(v,A)$ and constant envelope $1$. Let $\kappa_X = U_X /cb_X$ and $ \sigma_{\mathcal G} ^2 =\sup_{g\in \mathcal G , \, x\in S_X}  \Var(g(Y) |X=x)$.
Then, for all $n\geq 1$ and $1\leq k\leq n$ such that $\sigma_{\mathcal G} ^2   \kappa_X   k  \leq 2 n $, and all $\delta\in (0,1)$, it holds, with probability at least $1-\delta$:
 \begin{align*}
\hat Z_{n,k}   \leq  
 2  K_1 \sqrt{   ( d+1 +v )  \sigma_{\mathcal G} ^2 \kappa_X k     \log(  26 e^2  A n  / (\sigma _{\mathcal G} \delta)  } 
+   4 K_2 ( d+1 +v )    \log(  26e^2 A  n  / (\sigma _{\mathcal G} \delta)   )
\end{align*}
	with $K_1  $ and  $K_2 $ defined in the statement of Theorem \ref{vc_bound}.
\end{lemma}

 \textbf{Step (iii): End of the proof.} 
Combining Lemma \ref{prop:tau} and Lemma \ref{boundW} with the union bound (each is applied with $\delta/2 $ instead of $\delta$), we obtain that the event
\begin{align*}
 &E:=\left\{ \sup _{x\in S_X} \hat \tau_{n,k,x}    \leq \overline \tau_{n,k} \quad \text{and} \right.  \\  
& \left.  \hat Z_{n,k}   \leq   2 K_1 \sqrt{   ( d+1 +v )  \sigma_{\mathcal G} ^2 \kappa_X k \log \left(  52 e^2  \frac{A n } { \sigma _{\mathcal G} \delta}    \right)  } 
+   4K_2 ( d+1 +v )   \log \left( 52 e^2  \frac{A n } { \sigma _{\mathcal G} \delta} \right) \right\} 
\end{align*}
has  probability greater than $1-\delta$. On $E$, since $\hat \tau_{n,k,x} $ are smaller than $\overline \tau_{n,k}$, we find
 \begin{align*}
 \sup_{ g\in \mathcal G, \, x\in S_X }  \left |  k ^{-1} \sum_{i \in  N_{n,k}(x)}   (  g(Y_i )   -\mu_{X_i} (g)   )  \right | \leq k^{-1} \hat Z_{n,k}   .
\end{align*}
and using the previous upper bound on $\hat Z_{n,k}$, the result follows.
\qed

\section{Technical appendix}\label{app_tech}

\section{Auxiliary results}\label{seca}

In this section, we provide two useful auxiliary results. The following result is standard and known as the multiplicative Chernoff bound. The following version can be found in \cite{hagerup1990guided}.

\begin{theorem}\label{lemma=chernoff}
Let $(Z_i)_{i\geq 1}$ be a sequence of independent and identically distributed random variables valued in $\{0,1\}$. Set $\mu =  n \mathbb E [Z_1]$ and $S = \sum_{i=1} ^n Z_i $.   For any $\delta \in (0,1)$ and all $n\geq 1$, we have with probability at least $1-\delta$:
\begin{align*}
S \geq \left(1- \sqrt{ \frac{2 \log(1/\delta)  }{  \mu} } \right) \mu  .
\end{align*}
 In addition, for any $\delta \in (0,1)$ and $n\geq 1$, we have with probability at least $1-\delta$:
\begin{align*}
S \leq \left(1 +  \sqrt{ \frac{3 \log(1/\delta)   }{  \mu} }  \right) \mu  .
\end{align*}
\end{theorem}


The above Chernoff inequality is used in the proof of Theorem \ref{weakcv:kNN} with $Z_i = \ind_{B}(X_i)$ and $B$ is a given ball in $\mathbb R^d$. In the proof of Theorem \ref{consistenc:kNN}, we rely on an extended version of the Chernoff bound which holds uniformly over $B$ when restricted to the collection of closed balls
\begin{align*}
\mathcal  B = \{ B( x, \tau) \,:\, x\in \mathbb R ^d , \, \tau >0\} ,
\end{align*}
where $B(x,\tau ) $ is the closed ball with center $x$ and radius $\tau$. It takes different forms in the literature such as Theorem 2.1 in \cite{anthony1993result}; Theorem 15 in \cite{chaudhuri2010rates}; Theorem 1 in \cite{goix2015learning} and more recently Corollary 4.4 in \cite{lhaut2022uniform}. Using the bound in \cite{anthony1993result}, combined with the complexity results on the set of closed balls provided in \cite{wenocur1981some}, we obtain the following statement.

\begin{theorem}\label{lemma=prelim}
 Let $(X_i)_{i\geq 1}$ be a sequence of independent and identically distributed random variables valued in $\mathbb R^d$ with common distribution $P$. 
 For any $\delta > 0$ and $n\geq 1$, with probability at least $1 -\delta $:
\begin{align*}
n^{-1} \sum_{i=1} ^n  \ind_{B} (X_i) \geq  P(B)  \left( 1  - \sqrt{ \frac{   8d \log( 12n  /\delta   )   }{ nP( B) }}\right) ,\qquad \forall B \in \mathcal B.
\end{align*}
\end{theorem}
\begin{proof}
Theorem 2.1 in  \cite{anthony1993result} (see also Theorem 1.11 in  \cite{lugosi2002pattern}) states that for any Borelian set $\mathcal A$, $\delta > 0$ and $n\geq 1$, it holds with probability at least $1 -\delta $:
\begin{align*}
n^{-1} \sum_{i=1} ^n  \ind_{B} (X_i)   \geq  P(B)  \left( 1  - \sqrt{ \frac{ 4  ( \log(4/\delta ) + \log (S_{\mathcal A}(2n) )  ) }{ nP( B) }}\right) ,\qquad \forall B \in \mathcal A,
\end{align*}
where  $n \mapsto S_{\mathcal A}(n) $ stands for the shattering coefficient. Using that $ S_{\mathcal A}(n) \leq (n+1) ^{V} $ where $V$ is the VC dimension of $\mathcal A$ (see for instance Corollary 1.3 in \cite{lugosi2002pattern}) and the fact that the set of closed balls has a VC dimension equal to $d+1$   (Corollary 3.3 in \cite{wenocur1981some}), we find that $S_{\mathcal B}(2n)\leq (2n+1)^{d+1}$. We obtain that with probability at least $1-\delta$
\begin{align*}
n^{-1} \sum_{i=1} ^n  \ind_{B} (X_i)  \geq  P(B)  \left( 1  - \sqrt{ \frac{  4    (d+1) \log(4 (2n+1)  /\delta   )   }{ nP( B) }}\right) ,\qquad \forall B \in \mathcal B.
\end{align*}
Using that $4 (2n+1)  \leq 12n$ and $4 (d + 1 ) \leq 8d$ leads to the result.
\end{proof}

\section{Proof of Theorem \ref{vc_bound}}\label{secb}
{
The proof is inspired from  \cite{gineConsistencyKernelDensity2001}, proof of Proposition 2.2. Suppose that $|g|\leq 1$ and   $ P (g)  = 0$. Let $ W = \sup_{g\in \mathcal G} \left| \sum_{i=1} ^n   g (Z_i )  \right|  $. We have $W =  (W-\mathbb E W )  +   \mathbb E W $.
 Let $s = 2\mathbb E W + n\sigma^2$. From Bousquet's concentration inequality \citep{bousquet2002bennett},  Theorem 2.3, we have
$$P \left( W - \mathbb E W  \geq \sqrt{2s\log(1/ \delta)} + \log(1/ \delta) / 3 \right) \leq \delta .$$
Hence, using $\sqrt{a+b} \leq \sqrt a + \sqrt b $, it holds, with probability at least $1-\delta$,
$$ W - \mathbb E W  \leq 2 \sqrt{ \mathbb E W \log(1/ \delta)} + \sqrt{2 n\sigma^2 \log(1/ \delta)}+ \log(1/ \delta) / 3, $$
and invoking that $ 2ab \leq a ^2 + b^2$, we obtain 
$$ W - \mathbb E W  \leq    \mathbb E W   + \sqrt{2 n\sigma^2 \log(1/ \delta)}+ 4 \log(1/ \delta) / 3. $$
As a consequence we have with probability $1-\delta$,
$$ W\leq  2 \mathbb E W   + \sqrt{2 n\sigma^2 \log(1/ \delta)}+ 4 \log(1/ \delta) / 3.$$
We now apply the previous inequality to $(g -  P(g) ) /2U$. It gives, with probability at least $1-\delta$,
\begin{align*}
 \sup_{g\in \mathcal G} \left| \sum_{i=1} ^n   \{g (Z_i )  -   P(g) \} \right|   \leq  2\mathbb E [ \sup_{g\in \mathcal G} \left| \sum_{i=1} ^n   \{g (Z_i )  -   P(g)  \} \right|] + \sqrt{2 n\sigma^2 \log(1/ \delta)}+ 8U \log(1/ \delta) / 3.
\end{align*} 
Let $(\eta_1 ,\ldots, \eta_n) $ denote a collection of independent Rademacher variables, that is, $\mathbb P (\eta_i = +1) = \mathbb P (\eta_i = -1) = 1/2$ for all $i$. Let $(Z_1',\ldots, Z_n') $ be a collection of random variables independent from $(Z_1,\ldots, Z_n)$ but with the same distribution as $(Z_1,\ldots, Z_n)$. Let  $ \mathbb E_\eta $ denote the expectation over the Rademacher variables  while  considering the other variables as fixed. The Rademacher complexity is defined as
\begin{align*}
\hat R_n(\mathcal G')  & = \mathbb E_\eta [ \sup_{g'\in \mathcal G'} \left|  R_{g'} \right|  ],
\end{align*}
where  $R_{g'} = \sum_{i=1} ^n  \eta _i g'(Z_i,Z_i')  $ and $\mathcal G' =\{(z,z') \mapsto g(z) - g(z') \,:\, g\in \mathcal G\}$. Using the symmetrization Lemma, e.g., Lemma 7.3 in \cite{van2014probability}, we get,  with probability at least $1-\delta$,
\begin{align}\label{eq_bousq}
 \sup_{g\in \mathcal G} \left| \sum_{i=1} ^n   \{g (Z_i )  -   P(g) ] \} \right|   \leq  2 \mathbb E [ \hat R_n(\mathcal G') ] + \sqrt{2 n\sigma^2 \log(1/ \delta)}+ 8U \log(1/ \delta) / 3.
\end{align} 
An upper-bound on $  \mathbb E [ \hat R_n(\mathcal G') ]   $ is now needed. We start by deriving an upper-bound on $\hat R_n(\mathcal G') $. Remark that $ \sup _{g'\in \mathcal G'} | R_{g'}| = \sup_{g'\in \{\mathcal G' \cup -\mathcal G' \}}  R_{g'} $.
 Dudley's entropy integral bound \citep{dudley1967sizes}, as stated in the proof of Theorem 6.25 in  \cite{zhang2023mathematical}, gives that, with probability $1$,
\begin{align*}
\hat R_n(\mathcal G') & \leq  12 \sqrt n  \int_0 ^  {\sigma_n^\prime/2} \sqrt{ \log(   \mathcal N ( \{\mathcal G' \cup -\mathcal G' \},   L_2(\mathbb P'_n)  , \epsilon    ) )  } \diff \epsilon ,
\end{align*}
where $L_2(\mathbb P'_n) $ is the probability distribution $  (1/n) \sum_{i=1} ^n   \delta_{(Z_i, Z_i' )}   $ and $ \sigma_n^{'2} :=  \sup_{g'\in \mathcal G'}   P'_n ( g^{'2} )$.  Easy  manipulations on 
covering numbers gives that $  \mathcal N ( \{\mathcal G' \cup -\mathcal G' \},   L_2(\mathbb P'_n)  , \epsilon    ) \leq 2 \mathcal N ( \mathcal G' ,   L_2(\mathbb P'_n)  , \epsilon    )$ and using a 
 variable change, we obtain 
$$ \hat R_n(\mathcal G') \leq  12 (2U) \sqrt n \int_0 ^{  \sigma_n' / (4U) }  \sqrt{ \log(  2\mathcal N ( \mathcal G' ,   L_2(\mathbb P'_n),  \epsilon (2U)  ) ) } \diff \epsilon . $$ 
 Note that $\sigma_n' \leq (2U)$ by definition of the envelope. The underlying class $\mathcal G'$ satisfies the VC property with envelope $2U$ and VC parameter $(2v,A)$. The previous can be obtained as follows. Let $Q'$ be a probability measure on $S\times S$ and denote by $Q_1$ and $Q_2$ the respective marginals both defined on $S$. Introduce a $(U \epsilon) $-covering of $(\mathcal G,L_2(Q_1))$ (resp. $(\mathcal G,L_2(Q_2))$) with centres $g_k^{(1)}$ (resp. $g_j^{(2)}$), $k = 1,\ldots ,N_1$ (resp., $j = 1,\ldots ,N_2$), and notice that  $(z,z') \mapsto g_k(z) - g_j(z')$, $k = 1,\ldots ,N_1$,  $j = 1,\ldots ,N_2,$ are the centres of a $(2U\epsilon) $-covering  of $(\mathcal G ', L_2 (Q'))$. Therefore, it holds that 
  $$ \mathcal N ( \mathcal G ' ,   L_2(Q')  , \epsilon (2U)    )\leq   \mathcal N ( \mathcal G ,   L_2(Q_1)  , \epsilon U    ) \mathcal N ( \mathcal G ,   L_2(Q_2)  , \epsilon U    )\leq (A/\epsilon)^{2v} .$$
The previous is in particular valid for $Q' =\mathbb P_n'$ and we obtain
\begin{align*}
\hat R_n(\mathcal G')    \leq  12 (2U)\sqrt n  \int_0 ^ {  \sigma_n' / (4U) }  \sqrt{ \log(2) + 2 v \log( A/\epsilon ) }  \diff \epsilon  .
\end{align*}
Using the variable change $\epsilon = \sigma_n' / (4Us)$ gives
\begin{align*}
\hat R_n(\mathcal G')   & =  6 \sigma_n ' \sqrt{  n }   \int_{  1 } ^\infty   s^{-2}  \sqrt{  \log(2) + 2v\log(  A(4U) s /  \sigma_n'  ) }     \diff s.
    \end{align*}
Now since $A\geq 1$, $v\geq 1$,  and $\log( \sqrt 2 A(4U) / \sigma_n ' ) \geq  1$, we can write for all $s\geq 1$,
\begin{align*}
 \log(2) + 2v  \log(  A(4U)  s / \sigma_n'  ) &= \log(2) +  2v\log( A(4U)   /\sigma_n' )  +  2v  \log( s)\\
 & = 2 \log(\sqrt 2) +  2v\log( A(4U)   /\sigma_n' )  +   2v  \log( s)\\
 & \leq 2 v (  \log( \sqrt 2 A(4U)   /\sigma_n' ) +    \log( s))\\
 &\leq 2v \log( \sqrt 2 A(4U) / \sigma_n' )  ( 1 +  \log( s))
\end{align*} 
which, given that $\int_{  1 } ^\infty   s^{-2} \sqrt { \log(  s  ) } = \sqrt \pi / 2 $ and using $ \sqrt {a+b }\leq \sqrt a +\sqrt b$,  implies that
\begin{align*}
  \hat R_n(\mathcal G')  &\leq   6   \sqrt{  2  \sigma_n   ^{'2}nv  \log(  \sqrt 2 A(4U)   / \sigma_n'  ) }   \int_{  1 } ^\infty   s^{-2}  (  1 + \sqrt{ \log(  s  ) }  )    \diff s\\
  & \leq  6    ( 1  + \sqrt { \pi }  / 2 )   \sqrt{ 2 \sigma_n^{\prime 2} n v  \log( \sqrt 2 A(4U)   / \sigma_n'  ) }    \\
  &  =  C  \sqrt{ \sigma_n^{\prime 2} n v  \log(  2(A4U) ^2  / \sigma_n ^{\prime 2} )  }   .
\end{align*}
with $C = 12 \geq 6   (1 + \sqrt { \pi } / 2   )  $. Now we can proceed by taking the expectation with respect to the collection $(Z_1,Z_1'),\ldots, (Z_n,Z_n' ) $.
 Using (twice) Jensen inequality (functions $\sqrt x$ and $ a x\log(b/x) $ are both concave), we get
\begin{align*}
  \mathbb E [  \hat R_n(\mathcal G')  ]     &  \leq  C  \sqrt{ \mathbb E [ \sigma_n^{\prime 2}  n v  \log(  2 (4AU)^2 / \sigma_n^{\prime  2} )  ] }\\
  & \leq C  \sqrt{ \mathbb E [ \sigma_n^{\prime  2} ]  n v  \log(  2  ( 4AU)^2 /\mathbb E [  \sigma_n^{\prime  2} ]  )   }\\
  &\leq C  \sqrt{ \mathbb E [ \sigma_n^{\prime  2} ]  n v  \log(  e ( 6AU)^2 /\mathbb E [  \sigma_n^{\prime 2} ]  )   }
  \end{align*}
 From Corollary 15 in \cite{massart2000constants}, we obtain 
\begin{align*}
 \mathbb E [ \sigma_n^{\prime  2}  ] & \leq  \sigma ^{\prime 2} + 8(2U) n^{-1}   \mathbb E [  \hat R_n(\mathcal G')  ]  ,
\end{align*}
where $\sigma ^{\prime  2} = \sup_{g\in \mathcal G}  E [ (g(X_1 ) - g(X_1'))^2] = 2\sigma^2 $. 
Moreover,
using the envelope property, it follows that
$  \sigma ^{\prime  2} + 8(2U) n^{-1}   \mathbb E [  \hat R_n(\mathcal G')  ]  \leq  34U^2 $.
Since $x\mapsto x \log(  a / x )$ is increasing on $(0,   a /e ) $, and because $34 U^2 \leq e ( 6AU)^2 /e  $,
 we finally get, 
 \begin{align*}
   \mathbb E [  \hat R_n(\mathcal G')  ]   &\leq C  \sqrt{   n v  (  2\sigma ^2 + 16U n^{-1} \mathbb E [  \hat R_n(\mathcal G')  ]  )  \log\left( \frac{ (6 \sqrt e AU) ^2    }{  2\sigma ^2 + 16U n^{-1} \mathbb E [  \hat R_n(\mathcal G')  ] }  \right)   }\\
   &\leq C  \sqrt{   n v  (  2\sigma ^2 + 16U n^{-1} \mathbb E [  \hat R_n(\mathcal G')  ]  )  \log\left( \frac{   (6 \sqrt e AU) ^2   }{  2\sigma ^2  }  \right)   }\\
  &\leq  C  \sqrt{   n v  (  2\sigma ^2 + 16U n^{-1} \mathbb E [  \hat R_n(\mathcal G')  ]  )  \log\left( ( \sqrt {3e} AU /  \sigma )^2    \right)   }\\
  &\leq C  \sqrt{   n v  (  4\sigma ^2 + 32U n^{-1} \mathbb E [  \hat R_n(\mathcal G')  ]  )  \log\left(   9 AU  /  \sigma     \right)   }.
 \end{align*}
Writing $x =  \mathbb E [  \hat R_n(\mathcal G')  ] $, the above means  that  $ x^2 \leq bx + c $ for appropriate values of $b,c$. Since the previous implies that $ x \leq  (b  +\sqrt  {b^2 + 4c} )/ 2 \leq  \sqrt c + b $, 
we then obtain
\begin{align*}
 \mathbb E [  \hat R_n(\mathcal G')  ] & \leq 2 C \sqrt{ n v  \sigma ^2  \log\left(  9 AU  /  \sigma     \right) }  + 32 C^2 v U     \log\left(  9AU  /  \sigma     \right)  .
 \end{align*}
Combining this bound with \eqref{eq_bousq} allows to obtain that with probability at least $1-\delta$, it holds
\begin{align*}
&\sup_{g\in \mathcal G} \left| \sum_{i=1} ^n   \{g (Z_i )  -  \mathbb E [g(Z) ] \} \right|  \\
 &\leq  4 C \sqrt{ v  n \sigma ^2   \log\left( 9AU /  \sigma     \right)   }  + \sqrt{2 n\sigma^2 \log(1/ \delta)} +  64 C^2 v U     \log\left(  9 AU /  \sigma     \right)   +8 U \log(1/ \delta) / 3\\
 &\leq  ( 4C   + \sqrt 2 ) \sqrt{  vn \sigma^2   \log(   9AU    /  (\sigma  \delta )  )      } +   (  64 C^2   \vee  (8/ 3) )   Uv   \log( 9 AU / (\sigma  \delta )   )  .
\end{align*} 
This implies the stated result.\qed

}

\section{Proofs of the Lemmas}\label{secc}

\subsection{Proof of Lemma \ref{prop:tau_x}}

Let $\gamma \in (0,1/4 )$ and define $ \tau _{n,k,x} ^{(\gamma)} = (1+2\gamma) ^{1/d}  \tau _{n,k,x} $. Because $f_X$ is positive and continuous at $x$, it holds for $k/n$ small enough,
$$ \inf_{z \in {B} (x,  \tau_{n,k,x} ^{(\gamma) } )} f_{X} (z)  / f_{X} (x) \geq  (1 - \gamma)  .$$ 
This yields 
\begin{align*}
 P (X\in {B} (x,  \tau _{n,k,x} ^{(\gamma)}  ) )  & \geq (1 - \gamma)    f_{X} (x) \int _{{B} (x,  \tau _{n,k,x} ^{(\gamma)}   ) } \,\diff z  \\
&=      (1 - \gamma)    f_{X} (x)   V_{d}  \tau _{n,k} ^{(\gamma)d}   \\
&= (1- \gamma )  (1+2\gamma)  \frac{ k  }{ n } \\
&\geq  (1+\gamma /2
 )  \frac{ k  }{ n } .
\end{align*}
Applying Theorem \ref{lemma=chernoff} with $Z_i = \mathds 1 _{ {B} (x, \tau _{n,k,x} ^{(\gamma)}   )} (X_i )$, we have with probability at least $1-\delta$,
\begin{align*}
 \sum_{i=1} ^n \mathds 1 _{ {B} (x,  \tau _{n,k,x} ^{(\gamma)}   ) } (X_i ) &\geq \left(1- \sqrt{ \frac{2 \log(1/\delta)  }{  \mu} } \right) \mu .
\end{align*}
with $\mu = nP (X\in {B} (x,  \tau _{n,k,x} ^{(\gamma)}  ) )$. In particular, with probability at least $1-1/k$:
\begin{align*}
 \sum_{i=1} ^n \mathds 1 _{ {B} (x,   \tau _{n,k,x} ^{(\gamma)}  ) } (X_i ) &\geq \left(1- \sqrt{ \frac{2\log(k )  }{  \mu } } \right)  \mu .
\end{align*}
Using that $\mu \geq (1+\gamma/2) k$, as demonstrated before, and taking $k$ large enough 
to ensure that the previous lower bound is an increasing function with respect to $\mu$, we obtain that with probability at least $1-1/k$,
\begin{align*}
 \sum_{i=1} ^n \mathds 1 _{ {B} (x,  \tau _{n,k,x} ^{(\gamma)}   ) } (X_i ) &\geq \left(1- \sqrt{ \frac{2 \log(k)  }{   (1+\gamma/2)  k} } \right) k (1+\gamma/2) .
\end{align*}
When $ k \gamma^2 \geq 8 (1+\gamma/2)  \log(k)  $, we obtain, with probability at least $1 - 1/k$,
\begin{align*}
\sum_{i=1} ^n \mathds 1 _{ {B} (x,  \tau _{n,k,x} ^{(\gamma)}   ) } (X_i ) & \geq k - (  \sqrt{ 2(1+\gamma/2 )   k  \log(k)} - \gamma k /2)\geq k.
\end{align*}
By definition of $\hat \tau _{n,k,x} $, the previous implies that  $\hat \tau _{n,k,x}  \leq   \tau _{n,k,x} ^{(\gamma)} $ with probability going to $1$.
In the same way, we obtain that $  \tau _{n,k,x} ^{(-\gamma)} \leq \hat \tau _{n,k,x}  $ happens with probability going to $1$.  It follows that the event
\begin{align*}
(1-  2 \gamma)   \leq ( \hat \tau _{n,k,x}  /   \tau _{n,k,x} ) ^d\leq (1+ 2\gamma) 
\end{align*}
has probability going to $1$.  But since $\gamma $ is arbitrary, it means that $  (\hat \tau _{n,k,x}/   \tau _{n,k,x})^d  $ converges to $1$ in probability.
\qed

\subsection{Proof of Lemma \ref{lemma:brecketing_cond}}

In the proof, we set $$\tilde \mu_{n,k,x} = \tilde  \mu_{x, (3/2)\tau_{n,k,x} ^d  }.$$ 
We will make use of several inequalities that are valid when $n $ is large enough in virtue of the continuity of $f_X(z)$ and $\mu_z(F^2(\cdot,z)) $ at $z=x$ and the fact that $k/n \to 0$. In particular, we take $n$ large enough so that $(3/2)\tau_{n,k,x} ^d < \delta$ and
\begin{align*}
\sup_{   \|x-z\|\leq (3/2)\tau_{n,k,x} ^d}   \mu_z(F^2(\cdot,z)) f_X(z)  \leq \sqrt 3  \mu_x(F^2(\cdot,x)) f_X(x),
\end{align*}
as well as 
\begin{align*}
\mu_x(F^2(\cdot,x) ) \leq \sqrt 3  \tilde \mu_{n,k,x} (F^2) \qquad \text{and} \qquad  (3/4) \leq (n/k) P (\ind _{ B (x,  (3/2)^{1/d} \tau_{n,k,x} )} ).
\end{align*}
Let $\epsilon \in (0,1) $ and $([\underline f_k,\overline f_k])_{k=1,\ldots, M_\epsilon}  $ be a collection of $(\epsilon \|F\|_{ L_2 ( \tilde\mu_{n,k,x}  ) } , L_2 ( \tilde  \mu_{n,k,x} ) )$-brackets covering $\mathcal F$. We can assume that $\overline f_k\leq F$. If this would not be the case, one would take $ \overline f_k \wedge F$ in place of $\overline f_k$.
Let $u_j = 1/2 + j \epsilon^2 $, $j=0,\ldots,  N_\epsilon$ with 
$N_\epsilon = \lfloor 1/\epsilon^2\rfloor +1$ and define, for all $j =1,\ldots, N_\epsilon$,
\begin{align*}
\underline {\ind}_j = \sqrt {\frac n k } \ind _{B(x, u_{j-1} ^{1/d} \tau_{n,k,x} ) } ,  \qquad \overline {\ind}_j =  \sqrt {\frac n k } \ind _{B(x, u_{j} ^{1/d}  \tau_{n,k,x} ) } .
\end{align*}
Let $f\in \mathcal F$ and $u \in [1/2,3/2]$. We can find $k\in \{1,\ldots, M_\epsilon\}$ and $j\in \{1,\ldots, N_\epsilon\}$ such that 
\begin{align*}
\underline f_k \underline {\ind}_j    \leq f  \ind _{B(x, u ^{1/d} \tau_{n,k,x} ) }  \leq \overline f_k \overline{ \ind} _j .
\end{align*}
Hence there are $M_\epsilon N_\epsilon$ brackets $[\underline f_k \underline {\ind}_j ,  \overline f_k \overline{ \ind} _j ]$ to cover $ {\mathcal F} _{n,k,x}  $. We now compute their size. Checking that $\int _ { u_{j-1}^{1/d} \leq  \| u \|\leq u_j^{1/d}  } \, \diff u = V_d  \epsilon^2$, and using the inequalities from the beginning of the proof, we find
\begin{align*}
P ( \overline f_k( \overline{ \ind} _j -  \underline {\ind}_j   ) ) ^2&\leq P ( F^2  ( \overline{ \ind} _j -  \underline {\ind}_j   )^2 ) \\
&= \left(\frac n k\right)  \int _{ u_{j-1}^{1/d} \leq  \tau_{n,k,x}^{-1} \|z-x\|\leq u_j^{1/d} } \mu_z(F^2(\cdot,z)) f_X(z) \,\diff z \\
&\leq \sqrt 3  \left(\frac n k\right)   \mu_x(F^2(\cdot,x)) f_X(x)  \tau_{n,k,x}^{d}   \int _ { u_{j-1} ^{1/d} \leq  \| u \|\leq u_j ^{1/d} } \, \diff u\\
& = \sqrt 3   \left(\frac n k\right)   \mu_x(F^2(\cdot,x)) f_X(x)  \tau_{n,k,x}^{d}   V_d \epsilon^2  \\
& = \sqrt 3   \mu_x(F^2(\cdot,x))   \epsilon^{2} \\
&\leq 3 \tilde \mu_{n,k,x} (F^2) \epsilon^2 .
\end{align*}
Using that 
$$ (3/4) \tilde \mu_{n,k,x} (F^2) \leq \tilde  \mu_{n,k,x} (F^2) \left(\frac n k\right)  P (\ind _{ B (x,  (3/2)^{1/d} \tau_{n,k,x} )} )  = P F_{n,k,x}^2,$$   
we get
\begin{align*}
P ( \overline f_k( \overline{ \ind} _j -  \underline {\ind}_j   ) ) ^2 &\leq  4   \epsilon^2  P (  F_{n,k,x}^2 )  .
\end{align*}
Moreover, it holds that
\begin{align*}
P ( (\overline f_k  -   \underline f_k  )    \underline {\ind}_j     )^2&\leq \left(\frac n k\right)  P ( (\overline f_k  -   \underline f_k  )   ^2  \ind _{ B (x,  (3/2)^{1/d}  \tau_{n,k,x}  ) } ) \\
&=  \mu_{n,k,x}  (\overline f_k  -   \underline f_k  )^2  \left(\frac n k\right)   P (   \ind _{ B (x,  (3/2)^{1/d}  \tau_{n,k,x}  ) } ) \\
&\leq \epsilon ^2 \mu_{n,k,x} ( F^2) \left(\frac n k\right)   P (   \ind _{ B (x,  (3/2)^{1/d}  \tau_{n,k,x}  ) } ) \\
&  = \epsilon ^2 P( F_{n,k,x} ^2) .
\end{align*}
In virtue of Minkowski's inequality, it follows that
\begin{align*}
\|   \overline f_k \overline{ \ind} _j -  \underline f_k \underline {\ind}_j    \|_{L_2(P)} \leq 3 \epsilon \|  F_{n,k,x}\|_{L_2(P)} .
\end{align*}
It implies that for all $\epsilon \in (0,1/3)$,
\begin{align*}
 \mathcal N _{[\,]} ( {\mathcal F}_{n,k,x}, L_2(P) , 3 \epsilon \|  F_{n,k,x}\| _{L_2(P) } ) 
 &\leq  M_\epsilon N_\epsilon\leq M_\epsilon  2\epsilon^{-2} ,
\end{align*}
with $M_\epsilon = \mathcal N _{[\,]} ( \mathcal F, L_2(\tilde \mu_{n,k,x} ) ,  \epsilon \| F\| _{L_2(\tilde \mu_{n,k,x} ) } )$.
Consequently, if $(\delta_n)_{n\geq 1}$ is a positive sequence that goes to $0$, we have, for $n$ large enough,
\begin{align*}
&\int _0 ^{\delta_n} \sqrt {\log  \mathcal N _{[\,]} ( {\mathcal F}_{n,k,x}, L_2(P) , 3 \epsilon \|  F_{n,k,x}\| _{L_2(P) } )  }\,  \diff \epsilon  \\
&\leq \int _0 ^{\delta_n} \sqrt {\log \mathcal N _{[\,]} ( \mathcal F, L_2( \tilde \mu_{n,k,x} ) ,  \epsilon \| F\| _{L_2(\tilde \mu_{n,k,x} ) } ) }  \diff \epsilon + \int _0 ^{\delta_n}  \sqrt { \log  (2 / \epsilon ^2 )}  \, \diff \epsilon\\
&\leq \sup_{Q\in \{ \tilde \mu_{x,u} \, : \, u\in (0,\delta)\}}  \int _0 ^{\delta_n}  \sqrt {\log \mathcal N _{[\,]} ( \mathcal F, L_2(Q ) ,  \epsilon \| F\| _{L_2(Q ) } ) }  \diff \epsilon + \int _0 ^{\delta_n}  \sqrt { \log( 2/\epsilon^2)  } \, \diff \epsilon.
\end{align*}
The first term goes to $0$ in virtue of \eqref{def:brack_cond}, making the required integrability condition easily satisfied.
\qed

\subsection{Proof of Lemma \ref{prop:tau}}
Under the assumptions on $k$, it holds that $\overline \tau_{n,k} \leq T$. By using \eqref{cond:reg1} and \eqref{cond:reg2}, we obtain that
\begin{align*}
P ( {B} (x, \overline \tau_{n,k} ) ) & =   \int_{{B} (x, \overline \tau_{n,k}  ) \cap S_X } f_{X}(z) \,\diff z \\
& \geq b_X \lambda( {B} (x, \overline \tau_{n,k}  ) \cap S_X  )  \\
&\geq cb_X \lambda( {B} (x, \overline \tau_{n,k}  )   ) = c b_X V_{d} \overline \tau_{n,k} ^{  d}  = 2k/n.
\end{align*}
By applying Theorem \ref{lemma=prelim}, it follows that, with probability at least $1-\delta$,
\begin{align*}
\forall x\in  S_X \qquad  \sum_{i=1} ^n \mathds 1 _{ B  (x, \overline \tau_{n,k}  ) } (X_i ) 
& \geq  n P (B  (x, \overline \tau_{n,k} ) )  \left( 1  - \sqrt{ \frac{   8d \log( 12n  /\delta   )   }{ n P (B  (x, \overline \tau_{n,k} ) ) }}\right) \\
&\geq 2k \left( 1  - \sqrt{ \frac{   8d \log( 12n  /\delta   )   }{2k }}\right) .
\end{align*}
Using that $8d \log( 12n / \delta)  \leq k$, it holds that, with probability at least $1-\delta$,
\begin{align*}
\inf _{x\in  S_X}  \sum_{i=1} ^n \mathds 1 _{ B  (x, \overline \tau_{n,k}  ) } (X_i ) \geq k .
\end{align*}
By definition of $\hat \tau _{n,k,x}$, on the event $\sum_{i=1} ^n \mathds 1 _{ B  (x, \overline \tau_{n,k}  ) } (X_i )  \geq k$ one has $\hat \tau _{n,k,x} \leq  \overline \tau _{n,k}$. Consequently,  $ \inf _{x\in  S_X}  \sum_{i=1} ^n \mathds 1 _{ B  (x, \overline \tau_{n,k}  ) } (X_i )  \geq k$ implies that $ \sup_{x\in  S_X  }\hat \tau _{n,k,x} \leq  \overline \tau _{n,k}$.
\qed


\subsection{Proof of Lemma \ref{boundW}}

We first prove the following VC preservation property.

\begin{lemma}\label{vc:pres}
Let $\mathcal G$ be a VC class defined on $S$ with parameters $(v,A)$ and constant envelope $1$. The class of functions defined on $S\times \mathbb R^d$, given by 
$$\mathcal F= \{ (y,x) \mapsto (g(y)   - \mu_x (g))  \, : \, g \in \mathcal G\},$$ 
is VC with parameters $(2v, A)$ and constant envelope $2 $.
\end{lemma}

\begin{proof} Let $Q$ be a probability measure on $S\times \mathbb R^d$. Let $Q_1$ (resp. $Q_2$) be the first (resp. second) marginal of $Q$ and define the probability measure $\tilde Q_1$ on $S$ as 
$ \tilde Q_1 (A)  =  \int \mu_x ( A)  Q_2 (   \diff x) $. Let $(g_k)_{k=1,\ldots, N} $ (resp. $(\tilde g_k)_{k=1,\ldots, N} $) be a cover of $\mathcal G$ of size $\epsilon   $ with respect to the probability $Q_1$ (resp. $\tilde Q_1$). We can find $k$ and $j$ in $\{1,\ldots, N\}$ such that
\begin{align*}
&\| g (Y) - \mu_X (g)  - ( g _k(Y) - \mu_X (\tilde g_j)  )\| _{L_2 ( Q) } \\
&\leq 
\| g (Y) -   g _k(Y) \| _{L_2 ( Q) }  + \| \mu_X (g)  -  \mu_X (\tilde g_j)  \| _{L_2 ( Q) }\\
& \leq \| g  -   g _k  \| _{L_2 ( Q_1) }  + \|  g   -  \tilde g_j  \| _{L_2 ( \tilde Q_1) }\\
& \leq 2 \epsilon,
\end{align*} 
which implies that $N^2$ balls are needed to cover $\mathcal F$. The definition of VC classes leads to the result.
\end{proof}

Another useful VC preservation property is the following one (see Lemma 12 in \cite{ausset2022empirical}).

\begin{lemma}\label{vc:pres2}
Let $\mathcal G_1$ be a VC class defined on $S_1$ with parameters $(v_1,A_1)$ and constant envelope $U_1$. Let $\mathcal G_2$ be a VC class defined on $S_2$ with parameters $(v_2,A_2)$ and constant envelope $U_2$. 
The class of functions defined on $S_1\times S_2$, given by 
$$\mathcal H= \{ (x_1,x_2) \mapsto g_1(x_1) g_2(x_2)     \, : \, g_1 \in \mathcal G_1 , \, g_2 \in \mathcal G_2  \},$$ 
is VC with parameters $(2(v_1+  v_2) , 2(A_1 \vee A_2))$ and constant envelope $U_1U_2 $.
\end{lemma}

The next lemma establishes the VC property for indicators of balls.

\begin{lemma}\label{vc:pres3}
Let $\mathcal B  = \{ B(x, \tau)  \, : \, x\in  \mathbb R^d , \, \tau>0\}  $. The class  $\ind_{ \mathcal B}$ is VC with parameters $(  2 (d+1),     e^2 )$ and envelope $1$.
\end{lemma}
\begin{proof}
{The class of sets containing all the closed balls has a VC dimension equal to $v = d+1$  \citep[Corollary 3.3]{wenocur1981some}.
 It follows from Corollary 1 in \cite{haussler1995sphere} that for all $\epsilon \in (0,1)$,
\begin{align*}
\mathcal M (\mathcal B , Q , \epsilon  ) \leq  e (v+1)  ( 2e / \epsilon ) ^{ v },
\end{align*}
where  $\mathcal M (\mathcal B , Q , \epsilon  )$ stands for the packing number (as defined in \cite[p.96]{wellner1996} or \cite[p.363]{boucheron2013concentration}) of the space $\mathcal B $ equipped with the pseudo-metric $(A,B) \mapsto Q(A\Delta B)$. 
 It is well known, e.g., \citep[p.96]{wellner1996},  that $ \mathcal N (\mathcal B , Q , \epsilon  )\leq \mathcal M (\mathcal B , Q , \epsilon  )$. Moreover it is easy to see that
$ \mathcal N (\mathcal B , Q , \epsilon  ) =   \mathcal N (\ind_ {\mathcal B} , L_1(Q ), \epsilon  ) $ where $ \ind_{ \mathcal B} =\{ \ind _{  B    } \, ,\,  B\in \mathcal B\}$. But since for any $f$ and $g$ valued in $\{0,1\}$, it holds $\int (g-f)^2 dQ = \int |g-f|dQ$, we have $\mathcal N (\ind_ {\mathcal B} , L_2(Q ), \epsilon  )  \leq \mathcal N (\ind_ {\mathcal B} , L_1(Q ), \epsilon ^2 ) $. We finally get
$$\mathcal N (\ind_ {\mathcal B} , L_2(Q ), \epsilon  )  \leq e (v+1)  (2e/ \epsilon^2 ) ^{ v }  \leq    (  e^2/ \epsilon ) ^{2v}	. $$
}
\end{proof}

The class $\mathcal G $ is VC with parameters $(v,A)$ and constant envelope $1$. As a consequence of Lemma \ref{vc:pres}, the set $ (y,x) \mapsto (g(y)   - \mu_{x}  (g) )$, $g\in \mathcal G$, is VC with parameter $(2v, A) $  and constant envelope $2$. Since the class of interest in $\hat Z_{n,k}$ is made of products between elements of $ \mathcal F $ (introduced in Lemma \ref{vc:pres}) and elements of $\ind_{ \mathcal B}$,  Lemma \ref{vc:pres2} combined with Lemma \ref{vc:pres3} implies that the resulting class is of VC-type with parameters $(2( d+1 +v ), 2 (A \vee e^2 ))$ and envelope $2$. 
The variance parameter $\sigma^2$ is deduced from the inequalities 
\begin{align*}
  {E} [   (  g(Y)  - \mu_X(g)  ) ^2 \ind _{  B  (x, \tau)  }( X)^2 ]
& \leq \sigma_{\mathcal G} ^2  E [   \ind _{   B  (x, \tau)  (X)    } ]\\
& \leq  \sigma_{\mathcal G} ^2  U_X V_{d}  \tau^{  d} \\
&\leq   \sigma_{\mathcal G} ^2 U_X V_{d}  \overline {\tau} _{n,k} ^{  d} \\
& = 2  \sigma_{\mathcal G} ^2 \kappa_X \frac{  k}{ n}  = :\sigma^2.
\end{align*}
Hence, it remains to verify that  $  2\sigma_{\mathcal G} ^2   \kappa_X  ( k / n ) \leq 4    $ to apply Theorem \ref{vc_bound} with $U=2$ and $\sigma $ given before so that using $ \kappa_X\geq 1$  and $k\geq 1$, we find that $(U/\sigma)^2 \leq 2n/\sigma _{\mathcal G} ^2 $ to upper bound the terms in the logarithms. This yields 
\begin{align*}
\hat Z_{n,k} &\leq 2 K_1 \sqrt{   ( d+1 +v )  \sigma_{\mathcal G} ^2 \kappa_X k  \log(  9A' \sqrt {2 n}  / (\sigma _{\mathcal G} \delta)    )} 
+   4 K_2  ( d+1 +v )    \log(  9A' \sqrt {2 n}  / (\sigma _{\mathcal G} \delta)  )
\end{align*}
with $A' = 2 ( A \vee e^2)$, holding with probability at least $1-\delta$. Using that  $A \vee e^2 \leq Ae^2$ and $ 9 \times 2 \times \sqrt{2 n} \leq 26 n$ leads to the stated result.
\qed

\section{Proof of Proposition \ref{prop:suffisient_cond}}\label{app:prop_suffisient_cond}

It holds that or all $|u|\leq \delta $ and $A\in \mathcal S$
\begin{align*} 
   \mu_{x,u}   ( A  ) \leq  \beta \mu(A)  .
\end{align*}
Use the continuity to have $\delta>0 $ small enough such that for all $|u|\leq \delta$, it holds 
$ (1/2)\| G\|_{L_2(\mu_x)}\leq \| G\|_{L_2(\mu_{x,u} )} \leq 2 \| G\|_{L_2(\mu_x)}$. Define $A = 2\beta \| G \|_{L_2(\mu)}  / \| G\|_{L_2(\mu_x)} $ and note in particular that $A \geq 1$.

Let $\epsilon\in (0,2) $,  $|u|\leq \delta$ and consider $[\underline g_k, \overline g_k]$, $k= 1,\ldots ,N_\epsilon$, an $ \epsilon \| G \| _ {L_2(\mu)}   $-bracketing of $\mathcal G$ with respect to $\mu$.  Using that $J( \mathcal G,  L_2(\mu  )  ,  \delta) < \infty$, $[\underline g_k, \overline g_k]$, $k= 1,\ldots ,N_\epsilon$ is also an $\epsilon \beta \| G \| _ {L_2(\mu)} $-bracketing of $\mathcal G$ with respect to $ \mu_{x,u} $. Using that 
\begin{align*}
 \beta\| G \| _ {L_2(\mu)} &= \| G\|_{L_2(\mu_{x,u})}  \frac{ \beta \| G \| _ {L_2(\mu)} }{ \| G\|_{L_2(\mu_{x,u})} } \\
 & \leq  \| G\|_{L_2(\mu_{x,u})}  \frac{ 2\beta \| G \| _ {L_2(\mu)} }{ \| G\|_{L_2(\mu_{x})} } =  \| G\|_{L_2(\mu_{x,u})}   A,
\end{align*}
we find that $[\underline g_k, \overline g_k]$, $k= 1,\ldots ,N_\epsilon$ is an $\epsilon A \| G\|_{L_2(\mu_{x,u})}  $-bracketing of $\mathcal G$ with respect to $ \mu_{x,u} $. Therefore, for all $\epsilon \in (0,2)$ and all $|u|\leq \delta$, we have
\begin{align*}
\mathcal N_{[\,]} \left(\mathcal G,  L_2(\mu_{x,u} )  ,  A\epsilon \| G \| _ {L_2(\mu_{x,u})} \right)   \leq \mathcal N_{[\,]} \left(\mathcal G,  L_2(\mu)  ,  \epsilon \| G \| _ {L_2(\mu)}   \right) ,
\end{align*}
or equivalently, for all $\epsilon \in (0,2A)$ all $|u|\leq \delta$,
\begin{align*}
\mathcal N_{[\,]} \left(\mathcal G,  L_2(\mu_{x,u} )  ,  \epsilon \| G \| _ {L_2(\mu_{x,u})} \right)   \leq \mathcal N_{[\,]} \left(\mathcal G,  L_2(\mu)  ,  (\epsilon/A) \| G \| _ {L_2(\mu)  }   \right) .
\end{align*}
Consequently, if $(\delta_n)_{n\geq 1} $ is a positive sequence converging to $0$, it holds for $n$ large enough,
\begin{align*}
&\int _ 0 ^{\delta _n} \sqrt{\log \left( \mathcal N_{[\,]} \left(\mathcal G,  L_2(\mu_{x,u} )  ,  \epsilon \| G \| _ {L_2(\mu_{x,u})} \right)   \right) } \, \diff\epsilon \\
&\leq  \int _ 0 ^{\delta _n} \sqrt{\log \left( \mathcal N_{[\,]} \left(\mathcal G,  L_2(\mu)  ,  (\epsilon /A) \| G \| _ {L_2(\mu)  }   \right) \right) } \, \diff\epsilon \\
& = A \int _ 0 ^{\delta _n/A} \sqrt{\log \left( \mathcal N_{[\,]} \left(\mathcal G,  L_2(\mu)  ,  \epsilon  \| G \| _ {L_2(\mu)  }   \right) \right) } \, \diff \epsilon,
\end{align*}
but the previous upper bound goes to $0$ by assumption.
\qed

\bigskip

\noindent{\textbf{Acknowledgments.}}
The author would like to express his gratitude to Johan Segers and Aigerim Zhuman for spotting a mistake in the statement of Lemma \ref{prop:tau} in an earlier version of the paper.

\bibliographystyle{chicago}
\bibliography{revision_F.bbl}

\end{document}